\documentclass[11pt]{article}

\usepackage{amsmath, amssymb, amsthm}
\usepackage[margin=2.5cm]{geometry}
\usepackage{graphicx, float, soul, xcolor}
\usepackage{soul}
\usepackage[breakable]{tcolorbox}
\usepackage[ruled,vlined]{algorithm2e}
\usepackage{stmaryrd}
\usepackage{mathpazo}   

\usepackage[doc,hhb,msm]{optional}
\usepackage{tikz}

\usepackage{nameref}
\usepackage{mathtools}
\usepackage{empheq}
\usepackage{comment}

\usepackage{stmaryrd} 

\usepackage{etoolbox}
\makeatletter
\patchcmd{\@maketitle}{\LARGE \@title}{\LARGE\bfseries\@title}{}{}

\renewcommand{\@seccntformat}[1]{\csname the#1\endcsname.\quad}
\makeatother

\definecolor{darkblue}{rgb}{0,0,.5}

\usepackage{hyperref}
\hypersetup{
	colorlinks=true,		
	linkcolor=darkblue,		
	citecolor=darkblue,		
	urlcolor=darkblue 		
}

\makeatletter
\def\th@plain{%
	\thm@notefont{}
	\itshape 
}
\def\th@definition{%
	\thm@notefont{}
	\normalfont 
}

\renewenvironment{proof}[1][\proofname]{\par
	\normalfont
	\topsep0\p@\@plus3\p@ \trivlist
	\item[\hskip\labelsep\itshape
	#1\@addpunct{.}]\ignorespaces
}{%
	\qed\endtrivlist
}
\makeatother

\newtheorem{theorem}{Theorem}[section]
\newtheorem{lemma}[theorem]{Lemma}
\newtheorem{corollary}[theorem]{Corollary}
\newtheorem{proposition}[theorem]{Proposition}
\newtheorem{fact}[theorem]{Fact}
\theoremstyle{definition}
\newtheorem{definition}[theorem]{Definition}

\newtheorem{remark}[theorem]{Remark}
\newtheorem{assumption}[theorem]{Assumption}

\usepackage[shortlabels,inline]{enumitem}
\setlist[enumerate]{nosep}

\allowdisplaybreaks

\definecolor{labelkey}{rgb}{0,0.08,0.45}
\definecolor{refkey}{rgb}{0,0.6,0.0}
\definecolor{Brown}{rgb}{0.45,0.0,0.05}
\definecolor{lime}{rgb}{0.00,0.8,0.0}
\definecolor{lblue}{rgb}{0.5,0.5,0.99}
\definecolor{OliveGreen}{rgb}{0,0.6,0}
\definecolor{tyrianpurple}{rgb}{0.4, 0.01, 0.24}

\colorlet{hlcyan}{cyan!30}

\makeatletter
\def\namedlabel#1#2{\begingroup
	\def\@currentlabel{#2}%
	\label{#1}\endgroup
}
\makeatother

\newcommand{\bx}{\ensuremath{\mathbf{x}}}
\newcommand{\bu}{\ensuremath{\mathbf{u}}}

\newcommand{\weakly}{\ensuremath{\:{\rightharpoonup}\:}}

\newcommand{\nnn}{\ensuremath{{n\in{\mathbb N}}}}

\newcommand{\thalb}{\ensuremath{\tfrac{1}{2}}}

\newcommand{\fenv}[1]%
{\ensuremath{\,\overrightarrow{\operatorname{env}}_{#1}}}
\newcommand{\benv}[1]%
{\ensuremath{\,\overleftarrow{\operatorname{env}}_{#1}}}

\newcommand{\scal}[2]{\left\langle{#1},{#2}  \right\rangle}

\newcommand{\exi}{\ensuremath{\exists\,}}
\newcommand{\zeroun}{\ensuremath{\left]0,1\right[}}
\newcommand{\RR}{\ensuremath{\mathbb R}}

\newcommand{\RP}{\ensuremath{\mathbb{R}_+}}
\newcommand{\RPP}{\ensuremath{\mathbb{R}_{++}}}

\newcommand{\dom}{\ensuremath{\operatorname{dom}}\,}

		\newcommand{\zer}{\ensuremath{\operatorname{zer}}}

		\newcommand{\Fix}{\ensuremath{\operatorname{Fix}}}
		\newcommand{\Id}{\ensuremath{\operatorname{Id}}}

		\newcommand{\pinf}{\ensuremath{+\infty}}

		\newcommand{\by}{\ensuremath{\mathbf{y}}}
		\newcommand{\bz}{\ensuremath{\mathbf{z}}}

		{\begin{list}{}{%
					\settowidth{\labelwidth}{\textrm{#1~}}%
					\setlength{\leftmargin}{\labelwidth+\labelsep}}}
			{\end{list}}
		
		\usepackage[capitalize,nameinlink]{cleveref}
		\crefname{equation}{}{equations}
		\crefname{chapter}{Appendix}{chapters}
		\crefname{item}{}{items}
		\crefname{enumi}{}{}

		\providecommand{\norm}[1]{\lVert#1\rVert}

		\newcommand{\bv}{\ensuremath{{\mathbf{v}}}}
		\newcommand{\bw}{\ensuremath{{\mathbf{w}}}}

		\providecommand{\weak}{\rightharpoonup}

		\providecommand{\RR}{\mathbb{R}}

		\providecommand{\dom}{\operatorname{dom}}

		\newcommand{\fix}{\ensuremath{\operatorname{Fix}}}

		\providecommand{\gra}{\operatorname{gra}}
		\providecommand{\Id}{\operatorname{{ Id}}}

		\providecommand{\rras}{\rightrightarrows}

		\providecommand{\fix}{\operatorname{Fix}}

		\providecommand{\Id}{\operatorname{Id}}

		\providecommand{\zer}{\operatorname{zer}}

			\providecommand{\RR}{\mathbb{R}}
			
			\providecommand{\C}{\mathcal{C}}

			\definecolor{myblue}{rgb}{0.9,0.9,0.98}

			\allowdisplaybreaks 
				
				\usepackage{relsize}

				\newcommand{\crefpart}[2]{%
					\hyperref[#2]{\namecref{#1}~\labelcref*{#1}~\ref*{#2}}%
				}

				\definecolor{myseagreen}{HTML}{3FBC9D}
				\usepackage[color=myseagreen]{todonotes}
				
				\usepackage[backend=biber, style=numeric, sortlocale=auto, natbib=true, doi=false, isbn=false, url=false,eprint=false,giveninits=true,maxnames=6]{biblatex}
				\newcommand{\tQ}{\tilde{Q}}
				\renewcommand{\zeta}{z}
				\newcommand{\be}{\ensuremath{\mathbf{e}}}
				\newcommand{\Qname}{fixed-point relocator}

				\newcommand{\Lip}{\mathcal{L}}
				\newcommand{\mL}{\mathcal{L}}
				
				\newcommand{\mA}{\mathcal{A}}
				\newcommand{\mE}{\mathcal{E}}
				\newcommand{\mM}{\mathcal{M}}
				\newcommand{\mN}{\mathcal{N}}
				
				\newcommand{\Z}{S}
				\newcommand{\mZ}{\mathcal{S}}
				\newcommand{\mP}{\mathcal{P}}
				
				\newcommand{\mR}{\mathcal{R}}

				\newcommand{\Span}{\text{span}\,}
				\newcommand{\bone}{\ensuremath{{\boldsymbol{1}}}}
				\newcommand{\vertiii}[1]{{\left\vert\kern-0.25ex\left\vert\kern-0.25ex\left\vert #1 
						\right\vert\kern-0.25ex\right\vert\kern-0.25ex\right\vert}}
				
\begin{document}
					
					\title{Relocated Fixed-Point Iterations with Applications to Variable Stepsize Resolvent Splitting}
					
					\author{
						Felipe Atenas\thanks{Centro de Modelamiento Matem\'atico (CNRS IRL2807), Santiago, Chile.
							E-mail:~\href{mailto:fatenas@cmm.uchile.cl}{fatenas@cmm.uchile.cl}}
						\and
						Heinz H.\ Bauschke\thanks{Mathematics, University
							of British Columbia,
							Kelowna, B.C.\ V1V~1V7, Canada. 
							E-mail:~\href{mailto:heinz.bauschke@ubc.ca}{heinz.bauschke@ubc.ca}}
						\and
						Minh N.\ Dao\thanks{School of Science,
							RMIT University, Australia
							E-mail:~\href{mailto:minh.dao@rmit.edu.au}{minh.dao@rmit.edu.au}}
						\and
						Matthew K.\ Tam\thanks{School of Mathematics \& Statistics,
							The University of Melbourne, Australia.
							E-mail:~\href{mailto:matthew.tam@unimelb.edu.au}{matthew.tam@unimelb.edu.au}} 
					}
					
					\date{\today}
					
					\maketitle

					\begin{abstract}
						In this work, we develop a convergence framework for iterative algorithms whose updates can be described by a one-parameter family of nonexpansive operators. Within the framework, each step involving one of the main algorithmic operators is followed by a second step which ``relocates'' fixed points of the current operator to the next. As a consequence, our analysis does not require the family of nonexpansive operators to have a common fixed point, as frequently assumed in the literature. Our analysis uses a parametric extension of the demiclosedness principle for nonexpansive operators.  As an application of our convergence results, we develop a version of the graph-based extension of the Douglas--Rachford algorithm for finding a zero of the sum of $N\geq 2$ maximally monotone operators, which does not require the resolvent parameter to be constant across iterations.
					\end{abstract}

					{\small

					}

					\section{Introduction} 
					Throughout this work we assume $X$ is a real Hilbert space with inner product $\scal{\cdot}{\cdot}$ and induced norm $\|\cdot\|$. We are interested in solving monotone inclusions of the form 
					\begin{equation} \label{e:gen-prob}
						\text{find~~} x \in X \text{~~such that~~} 0 \in (A+B)x,
					\end{equation}
					where $A$ and $B$ are maximally monotone operators on $X$ with $Z := \zer(A+B)\neq\varnothing$ as well as various extensions thereof (\emph{e.g.,} inclusions with more than two operators). 
					
					The \emph{Douglas--Rachford (DR)} algorithm \cite{DR,60DR,LM} is a distinguished splitting method used to solve~\eqref{e:gen-prob} by using the so-called \emph{resolvents} of the individual operators $A$ and $B$, without requiring any type of direct computation involving the operator $A+B$. Recall that the \emph{resolvent} of $A$ is the operator given by
					$J_{ A} := (I +  A)^{-1},$
					and the \emph{reflectent} of $A$ (also known as the \emph{reflected resolvent}) is the operator 
					$R_{ A} := 2 J_{ A} - \Id. $
					Given a stepsize $\gamma\in\RPP$ and an initial point $x_0\in X$, the DR algorithm for the ordered pair $(\gamma A,\gamma B)$ generates a sequence $(x_n)_{\nnn}$ according to $x_{n+1}:=T_{\gamma}x_n$ where $T_{\gamma}$ denotes the \emph{DR operator} defined by 
					\begin{equation}
						\label{e:0222d}
						T_\gamma := \Id-J_{\gamma A} + J_{\gamma B}R_{\gamma A}.
					\end{equation}
					In the standard convergence analysis for the DR algorithm, which is based on an interpretation as a fixed-point iteration, firmly nonexpansiveness of $T_{\gamma}$ yields the inequality
					\begin{equation}\label{e:dr-fne}
						(\forall \nnn)(\forall y_\gamma \in \fix T_{\gamma})\: \|x_{n+1} - y_\gamma\|^2+\|x_n-x_{n+1}\|^2 \leq \| x_n - y_\gamma \|^2,\end{equation}
					from which general theory \cite{MR211301} implies that $(x_n)_{\nnn}$ converges weakly to a point $x\in\Fix T_{\gamma}$ and that $J_{\gamma A}x$ solves \eqref{e:gen-prob}. In \eqref{e:dr-fne}, we emphasise that the fixed point $y_\gamma$ depends on the constant $\gamma$. Despite $J_{\gamma A}$ not being weakly continuous in general, with some additional work---leveraging demiclosedness of maximally monotone operators---it can also be shown that $(J_{\gamma A}x_n)_{\nnn}$ converges weakly to $J_{\gamma A}x$ \cite{Svaiter2011}. Convergence of the method has also been established in various nonmonotone settings \cite{A25,DP19,phan2016linear,li2016douglas,TP20}. 
					
					For many resolvent-based iterative methods such as the \emph{proximal point} and the \emph{forward-backward} algorithms, the fixed-point set of the underlying algorithmic operator and the set of solutions coincide. Consequently, the fixed-point set does not depend on algorithm parameters such as the stepsize. This property allows the stepsize to be varied across iterations whilst preserving an inequality in the form of \eqref{e:dr-fne} and hence the original convergence analysis holds without significant changes. In contrast, this is not the case for the DR algorithm whose fixed-point set $\Fix T_{\gamma}$ has a parametric dependence on $\gamma\in\RPP$ and need not coincide with the solution set $\zer(A+B)$. In fact, due to \cite[Proposition~26.1(i) \& Proposition~26.1(iii)(b)]{BC2017}, we have\begin{equation} 
						\Fix T_\gamma = \{ z + \gamma w: w \in Az, -w \in Bz\}. \label{e:Fix}
					\end{equation} As a consequence of \eqref{e:Fix}, the fixed-point sets of the DR operators for different parameters can be completely disjoint (see Remark~\ref{r:disjoint}) and thus \eqref{e:dr-fne} can no longer be used to directly analyse convergence of a variable stepsize version of of DR given by $(T_{\gamma_n}^nx_0)_{\nnn}$ for $(\gamma_n)$ in $\RR_{++}$.
					
					Although the stepsize parameter $\gamma\in\RPP$ in \eqref{e:0222d} does not usually play an important role in the DR algorithm \emph{convergence analysis} itself, its value is critical in methods that extend the DR algorithm to incorporate an additional operator through forward steps such as \emph{Davis-Yin splitting} \cite{DavisYin17,PedGidel18} and the \emph{forward-reflected-backward methods} \cite{MT20}. In these methods, the stepsize parameter is typically required to be sufficiently small so that the forward operator satisfies a Lipschitz-type inequality. Furthermore, as these methods reduce to the DR algorithm in special cases, their underlying operators' fixed-point sets inherit the parametric dependence on the stepsize. As a consequence, linesearches for these methods are not straightforward to develop and analyze. Another research direction where variable stepsize plays an important role is in extending \cite{guler1991convergence} on the convergence rate of the proximal point algorithm to the DR algorithm. 
					
					In order to tackle the challenge of allowing variable stepsizes for the DR algorithm, we develop an extension of the successive approximation scheme examined in \cite{MR211301}, allowing the use of a sequence $(\gamma_n)_{\nnn}$ of variable stepsizes. To do so, we introduce the notion of a \emph{\Qname} (Definition~\ref{d:Q-transport}). A key property of these operators is the following: given an operator $T_{\gamma}$ parametrized by a positive parameter $\gamma \in\RPP$, its fixed-point relocator operator, denoted $Q_{\delta\gets \gamma}$, is a  bijection which maps fixed points of $T_{\gamma}$ to fixed points of $T_{\delta}$ for a different parameter $\delta \in \RPP$. 
					
					Given  \Qname{s} $(Q_{\delta\gets \gamma})_{\gamma,\delta\in\RPP}$ for $(T_\gamma)_{\gamma\in\RPP}$, a stepsize sequence $(\gamma_n)_\nnn$  and initial point $x_0\in X$, we consider iterations of the form
					\begin{equation}\label{eq:RFP}
						(\forall \nnn)\quad x_{n+1} = Q_{\gamma_{n+1}\gets \gamma_n}T_{\gamma_n}x_n. 
					\end{equation}
					We refer to the family of methods defined this way as \emph{relocated fixed-point iterations.}
					For the DR operator, we show that a fixed-point relocator is given by
					\begin{equation} \label{e:DR-FPR}
						(\forall \gamma,\delta \in \RPP)(\forall x \in X)\; Q_{\delta\gets \gamma}x =  \tfrac{\delta}{\gamma}x +
						\big(1-\tfrac{\delta}{\gamma}\big)J_{\gamma A}x.
					\end{equation}
					The corresponding \emph{relocated DR algorithm} in the form of \eqref{eq:RFP} can therefore be expressed as 
					\begin{equation}\label{eq:RDR}
						\left\{\begin{aligned}
							w_{n} &= x_n-J_{\gamma_n A}x_n + J_{\gamma_n B}R_{\gamma_n A}x_n \\
							x_{n+1} &= \tfrac{\gamma_{n+1}}{\gamma_n}w_n + \left(1-\tfrac{\gamma_{n+1}}{\gamma_n}\right)J_{\gamma_n A}w_n.
						\end{aligned}\right.
					\end{equation}Note that, in its current form, \eqref{eq:RDR} appears to require an extra resolvent evaluation as compared to the standard DR algorithm. However, this can be circumvented using the equivalent implementation described in Algorithm~\ref{a:DR-2}. For further details, see Proposition~\ref{p:DR-N-reformulation}.
					\begin{algorithm}[!ht]
						\caption{Relocated DR algorithm for finding a zero of $A+B$. \label{a:DR-2}}
						\SetKwInOut{Input}{Input}
						\Input{Choose $x_0 \in X$ and a stepsize sequence $(\gamma_n)_\nnn$ in $\RPP$.}
						
						Set $z_0 = J_{\gamma_0 A}x_0$. 
						
						\For{$n=0,1,2,\dots$}{
							Step 1. Intermediate step. Compute \begin{equation} \label{DR:variable-stepsize-1}
								\left\{\begin{aligned}
									y_n & = J_{\gamma_nB}(2z_n - x_n) \\
									w_n & = x_n - z_n + y_n. \\
								\end{aligned}\right.
							\end{equation}
							
							Step 2. Next iterate. Compute
							\begin{equation} \label{DR:variable-stepsize-2}
								\left\{\begin{aligned}
									z_{n+1} & = J_{\gamma_n A}w_n \\
									x_{n+1} & = \dfrac{\gamma_{n+1}}{\gamma_n} w_n + \left( 1 - \dfrac{\gamma_{n+1}}{\gamma_n} \right) z_{n+1}. 
								\end{aligned}\right.
							\end{equation}
							
						}
					\end{algorithm}
					
					Several resolvent splitting algorithms with constant stepsizes have been developed to solve monotone inclusions, see \cite{LM,DavisYin17,ryu2020uniqueness,MT23,bredies2024graph} and references therein. An instance of a method allowing variable stepsizes for splitting methods that only admit resolvent evaluations is proposed  in \cite{Lorenz-Tran-Dinh}. Therein, the authors analyze a non-stationary extension of the DR algorithm by modifying the steps that define the DR operator  based on a reformulation for the linear case. The main modification happens in the evaluation of the resolvent of $B$. The authors provide convergence guarantees under assumptions on the behaviour of the sequence of stepsizes covered in our framework (see Remark~\ref{r:stepsize-cond}). In \cite{lorenz2025degenerate}, a variable stepsize DR method is proposed based on the degenerate variable metric proximal point algorithm. In particular, in \cite[Corollary 2.7]{lorenz2025degenerate}, under regularity assumptions on the asymptotic behavior of the preconditioners, the authors recover weak convergence of the iterates of the variable metric proximal point method to a solution to a maximally monotone inclusion. The structure of this result is similar to Theorem~\ref{t:abstractconv} below, although in the former the degenerate metric is variable, while in the latter the iteration operator itself varies and, therefore, the two approaches are conceptually different. Another approach is taken in \cite{cegielski2018regular} for solving inclusion problems using fixed-point iterations of the form $(\forall \nnn)\;x_{n+1} = T_nx_n$, where $(T_n)_{\nnn}$ is a family of operators such that $\cap_\nnn\Fix T_n \neq \varnothing$. The latter assumption is natural, for instance, for feasibility problems (see, e.g. \cite[Theorem 6.1]{cegielski2018regular}), but it may fail to hold for optimization problems, as  Remark~\ref{r:disjoint} shows below.

					In the present work, we take a different approach and propose to compose the original DR operator with a suitable \Qname~operator, while keeping the classical steps that define the operator in \eqref{e:0222d}.  For the stepsize rule defined in \cite{Lorenz-Tran-Dinh}, the non-stationary DR algorithm  in \cite[Eq. (12)]{Lorenz-Tran-Dinh}  and Algorithm~\ref{a:DR-2} coincide (see Remark~\ref{r:non-stat-DR}). In both cases,  when the sequence of stepsizes is constant, we retrieve the original DR algorithm. The perspective offered by our framework allows for the analysis of various other resolvent splitting algorithms for solving extensions of \eqref{e:gen-prob} involving $N\geq 2$ maximally monotone operators such as \cite{ryu2020uniqueness,MT23}.
					
					The remainder of this paper is organized as follows. In Section~\ref{s:prelim}, we introduce the notation used throughout the paper and recall mathematical preliminaries for use in our analysis.  In Section~\ref{s:resolvents-and-nonexpansive}, we present new calculus and continuity results for resolvents and the DR operator, and introduce an extension of the demiclosedness principle for nonexpansive operators which we term the \emph{parametric demiclosedness principle}. In Section~\ref{s:abstract-framework}, we introduce the notion of a \Qname, establish our main results regarding convergence of relocated fixed-point iterations, and provide some examples of this type of iterative method. 
					In Section~\ref{sec:resolvent-splitting}, we present further instances of the relocated fixed-point iteration framework, including a variable stepsize version of graph-based extensions of the DR algorithm introduced in ~\cite{bredies2024graph} for $N\geq2$ operators, and derive its convergence properties. We finish in Section~\ref{s:conclusion} with some concluding remarks. 
					
					\section{Preliminaries} \label{s:prelim}
					We denote the set of nonnegative real numbers by $\RP$ and the set of positive real numbers by $\RPP$. We say that a sequence of real numbers $(\delta_n)_\nnn$ is decreasing (resp.\ increasing)
					if $(\forall\nnn)$ $\delta_n\geq\delta_{n+1}$ (resp.\ $\delta_n\leq\delta_{n+1}$). This is also called 
					``nonincreasing'' (resp.\ ``nondecreasing'') in the literature.
					
					Given a real Hilbert space $X$, we use the notation $X^{\text{strong}}$ to denote the space $X$ endowed with the norm topology. Similarly, $X^{\text{weak}}$ denotes $X$ endowed with the weak topology. For a nonempty set $D \subseteq X$, $D^{\text{weak}}$ denotes the set $D$  with the topology induced by $X^{\text{weak}}$. Given a set $C \subseteq X$, $\overline{C}$ denotes the closure of $C$ in $X^{\text{strong}}$.  For $k\geq 1$, $X^k:=X\times\dots\times X$ denotes the $k$-fold product space.

					Given a set-valued operator $A: X \rras X$, its \emph{domain} is denoted by $\dom A := \{x\in X: Ax \neq\varnothing\}$, its \emph{graph} by $\gra A := \{ (x,y) \in X\times X: y \in Ax\}$, its set of \emph{zeros} by $\zer A := \{x \in X : 0 \in Ax\}$, and its set of \emph{fixed points} by $\fix A := \{x \in X: x \in Ax\}$. We also denote the \emph{inverse} of $A$ by $A^{-1}: X \rras X$ which is defined by $\gra A^{-1} := \{(y,x) \in X \times X: y \in Ax\}$. We say $A$ is \emph{monotone} if
					$(\forall (x,y), (u,v) \in \gra A ) \; \langle x-u, y - v\rangle \geq 0,$
					and \emph{maximally monotone} if it has no proper monotone extension. %
					
					In the study of convergence of iterative algorithms, closure properties of an underlying maximally monotone operator often play a crucial role. 
					
					\begin{definition}[Demiclosedness]\label{d:demi}
						Let $D \subseteq X$ be a nonempty weakly sequentially closed set. We say a set-valued operator $T: D \rras X$ is \emph{demiclosed} if $\gra T$ is sequentially closed in $D^{\text{weak}} \times X^{\text{strong}}$.
					\end{definition}
					
					\begin{fact}[{\cite[Corollary 20.38(ii)]{BC2017}}] \label{f:maximal-demiclosed}
						Let $A$ be a maximally monotone operator on $X$. Then $A$ is demiclosed.
					\end{fact}

					The sum of two maximally monotone operators need not be maximally monotone in general. However, under some additional assumptions \cite{BC2017}, the resulting sum operator is maximally monotone. We recall one setting of importance for our analysis in which this statement is true.
					
					\begin{fact}[{\cite[Corollary 25.5(i)]{BC2017}}] \label{f:sum-monotone}
						
						Suppose $A$ and $B$ are maximally monotone operators on $X$, with one of them having full domain. Then $A + B$ is maximally monotone.
						
					\end{fact}
					
					The convergence framework we develop in this work focuses on a particular type of iteration operators, the so-called \emph{nonexpansive} operators. A single-valued operator $T: X \to X$ is said to be:
					\begin{enumerate}
						\item \emph{nonexpansive} if
						$(\forall x,y \in X) \; \|Tx - Ty\| \leq \|x-y\|.$
						\item \emph{firmly nonexpansive}  if 
						$(\forall x,y \in X) \; \|Tx - Ty\|^2 + \|(\Id - T)x - (\Id - T)y\|^2 \leq \|x-y\|^2.$
						\item \label{it:avg} \emph{$\alpha$-averaged (nonexpansive)}  if $\alpha \in (0,1)$ and 
						$(\forall x,y \in X) \; \|Tx - Ty\|^2 + \frac{1-\alpha}{\alpha}\|(\Id - T)x - (\Id - T)y\|^2 \leq \|x-y\|^2.$
					\end{enumerate}
					Equivalent to \ref{it:avg}, $T$ is $\alpha$-averaged nonexpansive if there exists a nonexpansive operator $R$ such that $T = (1-\alpha)\Id + \alpha R$ \cite[Proposition 4.35]{BC2017}. Thus, in particular, an operator $T$ is firmly nonexpansive if and only if $T$ is $\frac{1}{2}$-averaged. 
					
					\begin{fact}[Properties of averaged operators] \label{f:averaged}
						
						Let $\alpha \in (0,1)$ and $T: X\to X$ be a nonexpansive operator. Then, the following hold.
						\begin{enumerate}
							\item \label{f:averaged-characterization}
							$T$ is an $\alpha$-averaged operator if and only if $(1-\frac{1}{\alpha}) \Id + \frac{1}{\alpha} T$ is nonexpansive.
							\item \label{f:averaged-scaling}
							$(\forall \lambda \in (0, \frac{1}{\alpha}))$ $T$ is an $\alpha$-averaged operator if and only if $(1-\lambda) \Id + \lambda T$ is $\lambda\alpha$-averaged.
						\end{enumerate}
						
					\end{fact}
					
					\begin{proof}
						Item \ref{f:averaged-characterization} is \cite[Proposition~4.35]{BC2017}, and item \ref{f:averaged-scaling} is \cite[Proposition~4.40]{BC2017}. 
					\end{proof}

					Given a nonempty set $C \subseteq X$, the \emph{normal cone} to $C$ at $x\in C$ is given by $N_C(x) := \{ v\in C: (\forall y \in C) \; \langle v, y - x \rangle \leq 0 \}$, and if $x \notin C$, $N_C(x) := \varnothing$ . If $C \subseteq X$ is a nonempty closed convex set, then $N_C$ is maximally monotone. The \emph{projection} operator onto $C$ is  $ P_{C} = J_{N_C}$.
					
					The following fact summarizes important properties of resolvents and averaged operators. 
					\begin{fact}[Properties of resolvents]  \label{f:resolvents}
						Suppose $A$ is a maximally monotone operator on $X$. Then, the following statements hold true.
						\begin{enumerate}
							\item \label{f:resolvents-nonexpansive}
							$(\forall \alpha \in \RPP )$ $J_{\alpha A}$ and $\Id - J_{\alpha A}$  are well-defined, have full domain, and are firmly nonexpansive and maximally monotone.
							\item \label{f:resolvents-stepsize-trick}
							$(\forall \gamma, \lambda \in \RPP )(\forall x \in X)$ $J_{\gamma A}x = J_{\lambda\gamma A}(\lambda x + (1-\lambda)J_{\gamma A}x)$.
							\item \label{f:resolvents-lip}
							$(\forall \gamma, \lambda \in \RPP )(\forall x \in X)$ $\| J_{\gamma A}x - J_{\lambda\gamma A} x \| \leq |1-\lambda| \| J_{\gamma A}x - x\|$.
							\item \label{f:resolvents-asymptotic}
							$(\forall x \in X)$ $J_{\gamma A}x \to P_{\overline{\dom A}}x $ as $\gamma \downarrow 0$.
						\end{enumerate}
					\end{fact}
					
					\begin{proof}
						The properties in \ref{f:resolvents-nonexpansive} follow from \cite{Minty}, \cite[Corollary~23.9]{BC2017} and \cite[Corollary~23.11(i)]{BC2017}. Moreover, items \ref{f:resolvents-stepsize-trick}--\ref{f:resolvents-asymptotic} are, respectively,  \cite[Proposition~23.31(i)]{BC2017}, \cite[Proposition~23.31(iii)]{BC2017}, and \cite[Theorem 23.48]{BC2017}.
						
					\end{proof}

					In the following facts, we recall the standard convergence result of the DR algorithm.

					\begin{fact}[Convergence of the DR algorithm]
						\label{f:DR-convergence}
						Suppose that $A,B$ are maximally monotone operators on $X$ 
						such that $Z := \zer(A+B)\neq\varnothing$. Let $\gamma \in \RPP$ be a stepsize parameter, and $T_\gamma$ be  the associated DR operator defined in \eqref{e:0222d}. Then, given an arbitrary initial point $ x \in X$, the sequence $(T_{\gamma}^nx)_{\nnn}$ converges weakly to a point $\bar x \in \Fix T_\gamma$, and the sequence $\big(J_{\gamma A}T_{\gamma}^nx)_{\nnn}$ converges weakly to  $ J_{\gamma A}\bar{x} \in Z$.
					\end{fact}
					
					\begin{proof}
						See, e.g., \cite[Theorem 1]{LM} and \cite[Theorem 1]{Svaiter2011}.
					\end{proof}
					
					In Fact~\ref{f:DR-convergence}, the weak limit of the sequence $(T_{\gamma}^nx)_{\nnn}$ is generally not a solution to problem \eqref{e:gen-prob}. The sequence whose weak limit corresponds to a solution is $(J_{\gamma A}T_{\gamma}^nx)_{\nnn}$, also known as the \emph{shadow sequence}. Moreover, the sequence given by $(J_{\gamma B}R_{\gamma A}T_{\gamma}^nx)_{\nnn}$ also converges weakly to a solution to problem \eqref{e:gen-prob}.
					
					The following is a fundamental result for proving convergence of iterative algorithms. We will use it in our analysis.
					
					\begin{fact}[Robbins--Siegmund]
						\label{f:RS}
						Let $(\alpha_n)_\nnn,(\beta_n)_\nnn,(\varepsilon_n)_\nnn \subseteq \RP$ be sequences
						such that 
						$\sum_\nnn\varepsilon_n<\pinf$ and 
						$(\forall\nnn)\; \alpha_{n+1} \leq 
						(1+\varepsilon_n)\alpha_n -\beta_n.$
						Then $(\alpha_n)_\nnn$ converges and 
						$\sum_\nnn \beta_n<\pinf$.
					\end{fact}
					\begin{proof}
						This is a special deterministic case of 
						a classical probabilistic result by Robbins and Siegmund \cite{RoSi}. 
						See also \cite[Lemma~11 on page~50]{Polyak}, 
						\cite[Lemma~3.1]{Comb}, \cite[Lemma~5.31]{BC2017}, as well as
						\cite[Theorem~3.2]{NePo} for a very recent quantitative version. 
					\end{proof}
					
					The last component of our analysis relies on a generalization of the  \emph{Fej{\'e}r monotonicity} property for sequences called the \emph{Opial property}~\cite{arakcheev2025opial}.
					
					\begin{definition}[Opial sequence]
						Let $C$ 
						be a nonempty subset of $X$. A sequence $(x_n)_\nnn$  in $X$ is said to be \emph{Opial} 
						with respect to 
						$C$ if 
						$(\forall c\in C)\; \lim_{n\to\infty}\norm{x_n-c}\; $
						\text{exists.}
					\end{definition}
					
					\begin{fact}[{\cite[Proposition~2.1, Corollary~2.3]{arakcheev2025opial},\cite[Lemma 5.2]{peypouquet2015convex}}]\label{f:Opial}
						Let $(x_n)_\nnn$ be a sequence in $X$ which is Opial with respect to a nonempty set $C\subseteq X$. Then, the following statements are true.
						\begin{enumerate}
							\item \label{f:Opial-bounded}
							$(x_n)_\nnn$ is bounded. 
							\item \label{f:Opial-cluster}
							$(x_n)_\nnn$ converges weakly to a point in $C$ if and only if all weak cluster points of $(x_n)_\nnn$ lie in~$C$.
						\end{enumerate}
					\end{fact}
					
					\section{Properties of resolvents and nonexpansive operators}  \label{s:resolvents-and-nonexpansive}

					In this section, we derive a number of new properties of resolvents and explore their consequences for the Douglas--Rachford operator. The insights of this section will form the basis and motivating examples for the framework developed in the subsequent sections. 
					
					\subsection{Properties of resolvents}
					\label{sec:resolvents}
					
					In this section, we first present a resolvent rule that yields a bijection between fixed-point sets of the DR operator as a function of the stepsize. This rule allows us to define an operator that transforms fixed points of the DR operator as a function of stepsize. We also establish continuity properties of the map $(x, \gamma) \mapsto J_{\gamma A}x$ for a given maximally monotone operator $A$. 
					
					\begin{lemma}
						\label{l:bijection}
						Let $A$ be maximally monotone on $X$, 
						and let $\alpha,\beta\in\RPP$. 
						Then 
						\begin{equation}
							\label{e:0222a}
							J_{\beta A}\Big(\tfrac{\beta}{\alpha}\Id 
							+ \big(1-\tfrac{\beta}{\alpha}\big)J_{\alpha A} \Big) = J_{\alpha A} 
						\end{equation}
						and 
						\begin{equation}
							\label{e:0222b}
							\Big(\tfrac{\alpha}{\beta}\Id 
							+\big(1-\tfrac{\alpha}{\beta}\big)J_{\beta A} \Big)
							\Big(\tfrac{\beta}{\alpha}\Id 
							+\big(1-\tfrac{\beta}{\alpha}\big)J_{\alpha A} \Big) 
							= \Id = 
							\Big(\tfrac{\beta}{\alpha}\Id 
							+\big(1-\tfrac{\beta}{\alpha}\big)J_{\alpha A} \Big)
							\Big(\tfrac{\alpha}{\beta}\Id 
							+\big(1-\tfrac{\alpha}{\beta}\big)J_{\beta A} \Big). 
						\end{equation}
						Consequently, 
						\begin{equation}
							\label{e:0222c}
							\tfrac{\alpha}{\beta}\Id 
							+\big(1-\tfrac{\alpha}{\beta}\big)J_{\beta A} 
							\colon X\to X \text{ is a bijection with inverse }
							\tfrac{\beta}{\alpha}\Id 
							+\big(1-\tfrac{\beta}{\alpha}\big)J_{\alpha A}. 
						\end{equation}
					\end{lemma}
					\begin{proof}
						Applying Fact~\ref{f:resolvents}\ref{f:resolvents-stepsize-trick} with $\gamma=\alpha$ and $\lambda = \beta/\alpha$
						yields \eqref{e:0222a}. 
						Using \eqref{e:0222a} in the second identity, we deduce 
						\begin{align*}
							\Big(\tfrac{\alpha}{\beta}\Id 
							+\big(1-\tfrac{\alpha}{\beta}\big)J_{\beta A} \Big)
							\Big(\tfrac{\beta}{\alpha}\Id 
							+\big(1-\tfrac{\beta}{\alpha}\big)J_{\alpha A} \Big)
							&= \Id + \Big(\tfrac{\alpha}{\beta}-1\Big)J_{\alpha A}
							+ \Big(1-\tfrac{\alpha}{\beta}\Big)J_{\beta A}%
							\Big(\tfrac{\beta}{\alpha}\Id 
							+\big(1-\tfrac{\beta}{\alpha}\big)J_{\alpha A} \Big)\\
							&=\Id,
						\end{align*}
						and so the left identity in \eqref{e:0222b} holds. 
						To obtain the right identity in \eqref{e:0222b},
						we use the left identity, with the roles of $\alpha$ and $\beta$ 
						interchanged. 
						Finally, \eqref{e:0222c} follows from \eqref{e:0222b}.
					\end{proof}
					
					Lemma~\ref{l:bijection} provides a bijective map useful for computing resolvents with stepsizes only requiring linear algebra operations and the knowledge of the resolvent for one fixed parameter. In the next result, we study further properties of this bijection.
					
					\begin{proposition}
						\label{p:Q}
						Let $A$ be maximally monotone on $X$, and let $\alpha,\beta\in\RPP$.
						Consider the operator defined by 
						\begin{equation*}
							Q_{\beta\gets \alpha} := \tfrac{\beta}{\alpha}\Id 
							+\big(1-\tfrac{\beta}{\alpha}\big)J_{\alpha A}.
						\end{equation*}
						Then:
						\begin{enumerate}
							\item 
							\label{r:Qi}
							If $\beta <\alpha$, then $Q_{\beta\gets \alpha}$ 
							is $(\thalb-\tfrac{\beta}{2\alpha})$-averaged and hence firmly nonexpansive.
							\item 
							\label{r:Qii}
							If $\beta \geq \alpha$, then 
							$Q_{\beta\gets \alpha}$ is monotone,
							$\tfrac{\beta}{\alpha}$-Lipschitz continuous, and
							\begin{equation}\label{eq:Qab 2}
								(\forall z,\bar{z}\in X)\quad \| Q_{\beta\gets \alpha}z- Q_{\beta\gets \alpha}\bar{z}\|^2 + \bigl(\frac{\beta^2}{\alpha^2}-1\bigr)\|J_{\alpha A}z-J_{\alpha A}\bar{z}\|^2\leq \frac{\beta^2}{\alpha^2}\|z-\bar{z}\|^2. 
							\end{equation}
						\end{enumerate}
					\end{proposition}
					\begin{proof}
						\ref{r:Qi}~Since $J_{\alpha A}$ is firmly nonexpansive from Fact~\ref{f:resolvents}\ref{f:resolvents-nonexpansive} and $0<1-\frac{\beta}{\alpha}< 1<2$, the result follows from Fact~\ref{f:averaged}\ref{f:averaged-scaling} with $\alpha=1/2$ and $\lambda=1-\beta/\alpha$.
						
						\ref{r:Qii}~Assume $\beta \geq \alpha$ and let $z,\bar{z}\in X$. Since $J_{\alpha A}$ is firmly nonexpansive, we have
						$  \|J_{\alpha A}z-J_{\alpha A}\bar{z}\|^2\leq \langle z-\bar{z},J_{\alpha A}z-J_{\alpha A}\bar{z}\rangle. $
						Using this inequality and noting that $1-\frac{\beta}{\alpha} \leq 0$, we deduce
						\begin{align*}
							&\| Q_{\beta\gets \alpha}z- Q_{\beta\gets \alpha}\bar{z}\|^2\\
							&= \left\|\frac{\beta}{\alpha}(z-\bar{z})+\left(1-\frac{\beta}{\alpha}\right)(J_{\alpha A}z-J_{\alpha A}\bar{z})\right\|^2 \\
							&= \frac{\beta^2}{\alpha^2}\|z-\bar{z}\|^2 + \left(1-\frac{\beta}{\alpha}\right)^2\|J_{\alpha A}z-J_{\alpha A}\bar{z}\|^2 + 2\frac{\beta}{\alpha}\left(1-\frac{\beta}{\alpha}\right)\langle z-\bar{z},J_{\alpha A}z-J_{\alpha A}\bar{z}\rangle \\
							&\leq \frac{\beta^2}{\alpha^2}\|z-\bar{z}\|^2 + \left(1-\frac{\beta}{\alpha}\right)^2\|J_{\alpha A}z-J_{\alpha A}\bar{z}\|^2 + 2\frac{\beta}{\alpha}\left(1-\frac{\beta}{\alpha}\right)\|J_{\alpha A}z-J_{\alpha A}\bar{z}\|^2\\
							&= \frac{\beta^2}{\alpha^2}\|z-\bar{z}\|^2 + \left(1-\frac{\beta^2}{\alpha^2}\right)\|J_{\alpha A}z-J_{\alpha A}\bar{z}\|^2,
						\end{align*}
						which yields \eqref{eq:Qab 2}, and hence $\tfrac{\beta}{\alpha}$-Lipschitz continuity of $Q_{\beta\gets \alpha}$. To see that $Q_{\beta\gets \alpha}$ is monotone, note the representation $Q_{\beta\gets \alpha}=\frac{\beta}{\alpha}(\Id-J_{\alpha A})+J_{\alpha A}$. Since both $\Id-J_{\alpha A}$ and $J_{\alpha A}$ are monotone in view of Fact~\ref{f:resolvents}\ref{f:resolvents-nonexpansive} and $\frac{\beta}{\alpha}\in\RPP$, monotonicity of $Q_{\beta\gets \alpha}$ follows. 
					\end{proof}
					
					\begin{remark}
						The Lipschitz bound in Proposition~\ref{p:Q}\ref{r:Qii}
						is sharp. Indeed, taking $A = N_{\{0\}}$ yields $Q_{\beta\gets \alpha} = \tfrac{\beta}{\alpha}\Id $ which is $\tfrac{\beta}{\alpha}$-Lipschitz continuous.
					\end{remark}
					
					Note that, due to \eqref{e:0222a}, the operator $Q_{\beta\gets \alpha}$ in Proposition~\ref{p:Q} satisfies $J_{\beta A}Q_{\beta\gets \alpha} = J_{\alpha A} .$ This identity will be used to propose a variable stepsize extension of the Douglas--Rachford algorithm. In order to analyze its convergence, we will also need the following continuity properties of the resolvent $J_{\gamma A}x$ as a function of both $x$ and $\gamma$, which were considered in \cite[Corollary 7]{friedlander2023perspective} in the context of convex subdifferentials in finite dimensional Euclidean spaces.

					\begin{proposition}\label{p:resnice}
						Let $A$ be maximally monotone on $X$, and define $S_A\colon X\times\RP \to X$ according to
						$$ S_A(x,\gamma) := \begin{cases} 
							J_{\gamma A}x & \gamma\in\RPP, \\
							P_{\overline{\dom A}}x & \gamma=0. \\
						\end{cases} $$
						Then $S_A$ is continuous on $X\times\RP$ and Lipschitz continuous on compact subsets of $X\times\RPP$.
					\end{proposition}
					\begin{proof}
						$S_A$ is well-defined because $\overline{\dom A}$ is closed and convex \cite[Theorem 1]{rockafellar1970virtual}. We first show that $S_A$ is continuous on $X\times\RP$. To this end, let $(x,\lambda)\in X\times\RP$ and suppose $(y,\gamma)\to(x,\lambda)$. First, assume $\lambda\in\RPP$. Using Facts~\ref{f:resolvents}\ref{f:resolvents-lip} and \ref{f:resolvents-nonexpansive}, we deduce that
						\begin{equation}\label{eq:Sxl}
							\begin{aligned}
								\|S_A(x,\lambda)-S_A(y,\gamma)\| 
								&\leq \|J_{\lambda A}x-J_{\gamma A}x\| + \|J_{\gamma A}x-J_{\gamma A}y\| \\
								&\leq \frac{\|J_{\lambda A}x-x\|}{\lambda}|\lambda-\gamma|+\|x-y\|\to0.
							\end{aligned}
						\end{equation}
						
						Next, assume $\lambda=0$. If $\gamma > 0$, using Fact~\ref{f:resolvents}\ref{f:resolvents-nonexpansive}, we deduce
						\begin{align*}
							\|S_A(x,0)-S_A(y,\gamma)\| 
							&\leq \|P_{\overline{\dom A}}x-J_{\gamma A}x\| + \|J_{\gamma A}x-J_{\gamma A}y\| \\
							&\leq \|P_{\overline{\dom A}}x-J_{\gamma A}x\|+\|x-y\|.
						\end{align*}
						And, if $\gamma = 0$, Fact~\ref{f:resolvents}\ref{f:resolvents-nonexpansive} gives $\|S_A(x,0)-S_A(y,0)\| 
						\leq \|P_{\overline{\dom A}}x-P_{\overline{\dom A}}y\| 
						\leq \|x-y\|.$ Combining these two subcases and noting Fact~\ref{f:resolvents}\ref{f:resolvents-asymptotic}, we deduce that $\|S_A(x,0)-S_A(y,\gamma)\|\to 0$.
						Thus, in all cases, $\lim_{(y,\lambda)\to(x,\gamma)}S_A(y,\lambda)=S_A(x,\gamma)$ which establishes continuity of $S_A$ on $X\times\RP$.
						
						To show that $S_A$ is Lipschitz continuous on compact subsets of $X\times\RPP$, let 
						$U \subseteq X\times\RPP$ be compact. Since $S_A$ is continuous and the continuous image of compact sets is compact, we have
						$ L := \max_{(x,\lambda)\in U}\frac{\|J_{\lambda A}x-x\|}{\lambda} < +\infty. $
						Then, using \eqref{eq:Sxl}, we obtain $(\forall (x,\lambda), (y,\gamma) \in U)$
						$$ \|S_A(x,\lambda)-S_A(y,\gamma)\| \leq \max\{L,1\}\left(|\lambda-\gamma|+\|x-y\|\right) \leq \sqrt{2}\max\{L,1\}\|(x,\lambda)-(y,\gamma)\|,  $$
						which shows that $S_A$ is Lipschitz on $U$.
					\end{proof}
					
					\subsection{Properties of the Douglas--Rachford operator}
					
					In this section, we explore a number of consequences of the resolvent properties derived in the previous section for the Douglas--Rachford operator $T_\gamma := \Id-J_{\gamma A} + J_{\gamma B}R_{\gamma A}$. First, we show that the bijective map of Lemma~\ref{l:bijection} also defines a bijection between fixed-point sets of the DR operator for different stepsize parameters, and show continuity properties of the operator $(x,\gamma) \mapsto T_{\gamma}x$. 
					
					\begin{theorem}
						\label{t:DRbijection}
						Let $A, B$ be maximally monotone on $X$, $T_{\gamma}$ be the DR operator defined in \eqref{e:0222d}, and $\alpha, \beta\in\RPP$.  Then 
						\begin{equation*}
							\Big(\tfrac{\beta}{\alpha}\Id 
							+\big(1-\tfrac{\beta}{\alpha}\big)J_{\alpha A}\Big)\Big|_{\Fix T_\alpha}
							\colon \Fix T_\alpha\to \Fix T_\beta
							\text{ is a bijection, with inverse }
							\Big(\tfrac{\alpha}{\beta}\Id 
							+\big(1-\tfrac{\alpha}{\beta}\big)J_{\beta A}\Big)\Big|_{\Fix T_\beta}. 
						\end{equation*}
					\end{theorem}
					\begin{proof}
						Let $x\in \Fix T_\alpha$.
						Then $z := J_{\alpha A}x = J_{\alpha B}R_{\alpha A}x \in \zer ( A + B)
						\;\text{ and }\; 
						R_{\alpha A}x = 2z-x.$ Now set $y := 
						\tfrac{\beta}{\alpha}x
						+ \big(1-\tfrac{\beta}{\alpha}\big)J_{\alpha A}x 
						= 
						\tfrac{\beta}{\alpha}x
						+ \big(1-\tfrac{\beta}{\alpha}\big)z. $ Using \eqref{e:0222a} thrice, we deduce that $ J_{\beta A}y =  J_{\alpha A}x = z $ and
						\begin{align*}
							J_{\beta A}y
							&= 
							J_{\beta A}\Big(\tfrac{\beta}{\alpha}\Id
							+ \big(1-\tfrac{\beta}{\alpha}\big)J_{\alpha A}\Big)x
							= J_{\alpha A}x = z = J_{\alpha B}R_{\alpha A}x  = 
							J_{\alpha B}(2z-x) \\
							&= J_{\beta B}\Big(\tfrac{\beta}{\alpha}(2z-x)+ 
							\big(1-\tfrac{\beta}{\alpha}\big)J_{\alpha B}(2z-x) \Big)= 
							J_{\beta B}\Big(\tfrac{\beta}{\alpha}(2z-x)+ 
							\big(1-\tfrac{\beta}{\alpha}\big)z \Big) \\
							&=
							J_{\beta B}\Big(2 z - \big(1-\tfrac{\beta}{\alpha}\big)z - 
							\tfrac{\beta}{\alpha}x \Big)
							= 
							J_{\beta B}(2z - y)= J_{\beta B}R_{\beta A}y.
						\end{align*}
						Hence $z=J_{\beta A}y = J_{\beta B}R_{\beta A}y$ 
						and so $y\in \Fix T_\beta$. 
						The conclusion now follows from Lemma~\ref{l:bijection}.
					\end{proof}
					
					Combining the previous theorem with the continuity result in Proposition~\ref{p:resnice} yields the following for the DR operator.
					
					\begin{corollary} \label{c:DR-op-bicont}
						Let $A$ and $B$ be maximally monotone on $X$, and $T_{\gamma}$ be the DR operator defined in \eqref{e:0222d}. The operator 
						\begin{equation*}
							X\times \RP \to X \colon (x,\gamma)\mapsto \begin{cases} 
								T_{\gamma}x & \gamma\in\RPP \\
								x-P_{\overline{\dom A}}x+P_{\overline{\dom B}}(2P_{\overline{\dom A}}x-x)& \gamma=0,\\
							\end{cases}
						\end{equation*}
						is continuous on $X\times \RP$ and locally Lipschitz continuous on compact subsets of $X\times\RPP$.
					\end{corollary}
					\begin{proof}
						Since $T_{\gamma}=\Id-J_{\gamma A}+J_{\gamma B}(2J_{\gamma A}-\Id)$, the result follows from Proposition~\ref{p:resnice}.
					\end{proof}
					
					The results proven in this section, in particular, Lemma~\ref{l:bijection}, Proposition~\ref{p:Q} and Theorem~\ref{t:DRbijection}, motivate the abstract framework we develop in Section~\ref{s:abstract-framework}. The operator $Q_{\beta\gets \alpha}$ defined in Proposition~\ref{p:Q} will provide the first example of what we will soon call a \Qname~ (Definition~\ref{d:Q-transport}).

					\subsection{Parametric demiclosedness principle}
					As mentioned in the introduction, resolvents --- and thus in particular the DR iteration operator in \eqref{e:0222d} --- are not necessarily weakly continuous. Browder's celebrated \emph{demiclosedness principle} \cite{MR230179} provides an alternative to assuming weak continuity of the operator in infinite-dimensional spaces. In this section, we investigate a parametric version of the demiclosedness principle. We start by recalling the principle as it is presented in \cite{BC2017}.

					\begin{fact}[Demiclosedness principle {\cite[Theorem~4.27]{BC2017}}] \label{f:demi}
						Let $D \subseteq X$ be a nonempty sequentially closed set, and let $T: D \mapsto X$ be a nonexpansive operator. Then $\Id - T$ is demiclosed.
					\end{fact}
					
					In their current forms, neither Definition~\ref{d:demi} or Fact~\ref{f:demi} apply to the Douglas--Rachford operator as a function of its stepsize. This motivates the following definition.
					
					\begin{definition}[Parametric demiclosedness]
						Let $D \subseteq X$ be a nonempty sequentially closed set, and $\Gamma \subseteq \RR$ be a nonempty set. We say an operator $G: D \times \Gamma \to X$ is \emph{parametrically demiclosed}~if $\gra G$ is sequentially closed in $(X^{\text{weak}} \times \Gamma) \times X^{\text{strong}}$.
					\end{definition} 
					
					The following result establishes a parametric generalisation of Fact~\ref{f:demi}.
					\begin{theorem}[Parametric demiclosedness principle] \label{t:edemi}
						Let $D \subseteq X$ be a nonempty weakly sequentially closed set,  $\Gamma\subseteq\RR$ be a nonempty closed set, and let $(T_{\gamma})_{\gamma\in\Gamma}$ be a family of nonexpansive operators from $D$ to $X$, such that $(\forall x \in D)\; \gamma \mapsto T_{\gamma}x$ is (strongly) continuous. Then the operator
						\begin{align*}
							D \times \Gamma &\to X\colon(x,\gamma)\mapsto (\Id-T_{\gamma})x
						\end{align*} is parametrically demiclosed.
					\end{theorem}
					\begin{proof}
						Let $(x_n, \gamma_n)_\nnn$ be a sequence in $D\times \Gamma$ such that $x_n \weakly x \in D$, $\gamma_n \to \overline{\gamma}\in\Gamma$, and $x_n - T_{\gamma_n}x_n =(\Id -T_{\gamma_n})x_n\to u$. We shall prove that $(\Id -T_{\overline{\gamma}})x =u$. Indeed, the continuity of $\gamma \mapsto T_{\gamma}x$ yields $T_{\overline{\gamma}}x - T_{\gamma_n}x\to0$. Thus, using the nonexpansiveness of $T_{\gamma_n}$ and the reverse triangle inequality, we have
						\begin{align*}
							\liminf_{n\to\infty} \|x_n - x\| & \geq \liminf_{n\to\infty} \|T_{\gamma_n}x_n - T_{\gamma_n}x\| \\
							& =  \liminf_{n\to\infty} \| (x_n - T_{\overline{\gamma}}x - u) + (u + T_{\gamma_n}x_n - x_n) + (T_{\overline{\gamma}}x - T_{\gamma_n}x) \| \\
							& \geq  \liminf_{n\to\infty} \| x_n - T_{\overline{\gamma}}x -u \| - \| u + T_{\gamma_n}x_n - x_n \| - \|T_{\overline{\gamma}}x - T_{\gamma_n}x \|  \\
							& =  \liminf_{n\to\infty} \| x_n - T_{\overline{\gamma}}x -u \| - \lim_{n\to\infty}\| u + T_{\gamma_n}x_n - x_n \| - \lim_{n\to\infty} \|T_{\overline{\gamma}}x - T_{\gamma_n}x \| \\
							& =  \liminf_{n\to\infty} \| x_n - T_{\overline{\gamma}}x -u \|.
						\end{align*}
						In view of \cite[Lemma 1]{MR211301}, we must have $x =T_{\overline{\gamma}}x +u$, or equivalently, $(\Id -T_{\overline{\gamma}})x =u$, which completes the proof.
					\end{proof}    
					
					\begin{remark}
						When $\Gamma$ is a singleton, parametric demiclosedness and standard demiclosedness coincide. Consequently, Theorem~\ref{t:edemi} recovers Fact~\ref{f:demi}.
					\end{remark}
					
					\section{Relocated fixed-point iterations}\label{s:abstract-framework}

					In this section, we propose a framework for sequences constructed as fixed-point iterations of a parametrized family of operators $(T_\gamma)_{\gamma \in \Gamma}$ composed with  operators $(Q_{\delta\gets \gamma})_{\gamma,\delta\in\Gamma}$ that relocate fixed-point sets. This setting covers resolvent-based algorithms with variable stepsizes, such as the DR algorithm, as discussed by the end of this section. 
					
					We show that given a parametric family of nonexpansive operators $(T_\gamma)_{\gamma\in\Gamma}$ and a sequence of stepsizes $(\gamma_n)_\nnn$ converging to some $\overline{\gamma}$, under appropriate assumptions, the corresponding relocated fixed-point iteration generates a sequence that converges weakly to a fixed point of $T_{\overline{\gamma}}$.
					
					\subsection{Convergence of relocated fixed-point iterations}  \label{s:abstract-framework-convergence}
					
					We start by defining the concept of \Qname. 
					
					\begin{definition}[Fixed-point relocator] \label{d:Q-transport}
						Let $\Gamma \subseteq \RPP$ be a nonempty set, and let $(T_\gamma)_{\gamma \in \Gamma}$ be a family of operators from $X$ to $X$. We say that the family of operators $(Q_{\delta\gets \gamma})_{\gamma,\delta \in \Gamma}$ on $X$ defines  \emph{\Qname{s}} for $(T_\gamma)_{\gamma \in \Gamma}$ with Lipschitz constants $(\Lip_{\delta\gets \gamma})_{\gamma,\delta \in \Gamma}$ in $[1,+\infty[$ if the following hold.
						\begin{enumerate}
							\item\label{a:trans_bijection}
							$(\forall \gamma, \delta \in \Gamma)$ $Q_{\delta\gets \gamma}|_{\Fix T_\gamma}$ is a bijection from $\Fix T_\gamma$ to $\Fix T_\delta$.
							\item \label{a:trans_cont} $(\forall \gamma \in \Gamma)(\forall x\in\Fix T_\gamma)$ $\Gamma\to X\colon \delta\mapsto Q_{\delta\gets \gamma}x$ is continuous. 
							\item\label{a:trans_semigroup}
							$(\forall \gamma, \delta, \varepsilon \in \Gamma)(\forall x \in \fix T_{\gamma})$ $Q_{\varepsilon\gets \delta}Q_{\delta\gets \gamma}x = Q_{\varepsilon\gets \gamma}x$.
							\item\label{a:trans_nonex-type}
							$(\forall \gamma, \delta \in \Gamma)$  $Q_{\delta\gets \gamma}$ is $\Lip_{\delta\gets \gamma}$-Lipschitz continuous.
						\end{enumerate}
					\end{definition}
					
					\begin{remark}[Fixed-point relocators outside the fixed-point set] \label{r:FPR} Given  \Qname{s} $(Q_{\delta\gets \gamma})_{\gamma,\delta\in\RPP}$ for $(T_\gamma)_{\gamma\in\RPP}$, observe that Definition~\ref{d:Q-transport}\ref{a:trans_bijection}--\ref{a:trans_semigroup} only need to hold on $\Fix T_\gamma$. Therefore, a \Qname~does not have to be unique. In fact, any Lipschitz operator $\tilde{Q}_{\delta\gets \gamma}$  such that $Q_{\delta\gets \gamma}|_{\Fix T_\gamma} = \tilde{Q}_{\delta\gets \gamma}|_{\Fix T_\gamma}$ also defines a \Qname~for $(T_\gamma)_{\gamma\in\RPP}$. We give an example of such a case in Proposition~\ref{p:Qname-MT}.
					\end{remark}
					
					\begin{remark}[Consequences of the \Qname{} definition] \label{r:FPR-inverse-and-identity}
						Given \Qname{s} $(Q_{\delta\gets \gamma})_{\gamma,\delta \in \RPP}$ for $(T_\gamma)_{\gamma\in\RPP}$, the following holds for all $\delta,\gamma \in \RPP$, \begin{equation*}
							Q_{\delta\gets\gamma} = \Id \text{~on~} \Fix T_{\gamma}, \text{~and~} (Q_{\delta\gets \gamma}|_{\Fix T_\gamma})^{-1} = Q_{\gamma\gets\delta}|_{\Fix T_\delta}.
						\end{equation*} Indeed, from Definition~\ref{d:Q-transport}\ref{a:trans_bijection},  for any $y \in \Fix T_{\delta}$,  there exists a unique $x \in \Fix T_{\gamma}$ such that $ y = Q_{\delta\gets\gamma}x$. From  Definition~\ref{d:Q-transport}\ref{a:trans_semigroup} with $\varepsilon = \delta$, $ y = Q_{\delta\gets \gamma}x = Q_{\delta\gets \delta}Q_{\delta\gets \gamma}x = Q_{\delta\gets \delta}y$. Hence $Q_{\delta\gets\delta} = I$ on $\Fix T_{\delta}$. Furthermore, Definition~\ref{d:Q-transport}\ref{a:trans_semigroup} with $\varepsilon = \gamma$ yields $Q_{\gamma\gets \delta}Q_{\delta\gets \gamma} = Q_{\gamma\gets \gamma} = I$ on $\Fix T_{\gamma}$. Interchanging the roles of $\delta$ and $\gamma$ in the last identity gives $Q_{\delta\gets \gamma}Q_{\gamma\gets \delta} = Q_{\delta\gets \delta} = I$ on $\Fix T_{\delta}$. Hence $(Q_{\gamma\gets \delta})_{| \Fix T_\delta}$ is the inverse of $(Q_{\delta\gets \gamma})_{\Fix T_{\gamma}}$.
					\end{remark}
					
					For the DR iteration operator \eqref{e:0222d}, an example of \Qname{s} is given in \eqref{e:DR-FPR}.  Lemma~\ref{l:example-Qname-DR} shows that this indeed the case by resorting to Proposition~\ref{p:Q} and Theorem~\ref{t:DRbijection}. In Proposition~\ref{p:Qname-graph-DR}, we define fixed-point relocators for graph-based DR algorithms.
					
					Within the context of Definition~\ref{d:Q-transport}, our goal is to analyze the sequence of successive approximations $(T_{\gamma}x)_{\gamma \in \Gamma}$, for a given $x \in X$, transformed by a \Qname. In order to do that, we consider the following assumption.
					\begin{assumption}
						\label{a:trans}
						Consider a nonempty set $\Gamma \subseteq \RPP$, and $(T_{\gamma})_{\gamma \in \Gamma}$ a family of nonexpansive operators on $X$ such that 
						\begin{enumerate}
							\item\label{a:trans_avg}
							$(\exi \alpha\in\zeroun)(\forall \gamma\in\Gamma)$ 
							$T_\gamma$ is $\alpha$-averaged. 
							\item\label{a:trans_bicont}
							$X\times \Gamma\to X\colon (x,\gamma)\mapsto T_{\gamma}x$ is continuous.
						\end{enumerate}    
					\end{assumption}
					
					Note that an instance of  operators $T_{\gamma}$ in Definition~\ref{d:Q-transport} satisfying Assumption~\ref{a:trans} is the DR operator in \eqref{e:0222d}, in view of Corollary~\ref{c:DR-op-bicont}.
					
					We are now ready for our main result. 
					
					\begin{theorem}[Convergence of relocated fixed-point iterations]
						\label{t:abstractconv}
						Let $\Gamma \subseteq \RPP$ be nonempty, $(T_{\gamma})_{\gamma\in\Gamma}$ be a family of nonexpansive operators on $X$, such that $(\forall \gamma \in \Gamma) \; \fix T_{\gamma} \neq\varnothing$, and $(Q_{\delta\gets \gamma})_{\gamma,\delta \in \Gamma}$ operators on $X$ be  \Qname{s} for $(T_{\gamma})_{\gamma \in \Gamma}$ with Lipschitz constant $(\Lip_{\delta\gets \gamma})_{\gamma,\delta \in \Gamma}$ in $[1,+\infty[$. Furthermore, let \begin{equation}
							\label{a:stepsize-series}
							(\gamma_n)_\nnn \text{~in~} \Gamma \text{~be a sequence that 
								converges to~} \overline{\gamma}\in \Gamma 
							\text{~and~}   \sum_\nnn (\Lip_{\gamma_{n+1}\gets \gamma_n} -1) < \pinf.
						\end{equation} 
						Let $x_0 \in X$ and generate a sequence $(x_n)_\nnn$ according to
						\begin{equation*}
							(\forall\nnn)\;x_{n+1} := Q_{\gamma_{n+1}\gets \gamma_n}T_{\gamma_n}x_n.
						\end{equation*}
						Then the following hold.
						\begin{enumerate}
							\item\label{t:abstractconv0} The sequence $(x_n)_\nnn$ is Opial with respect to $\Fix T_{\overline{\gamma}}$.
							\item 
							\label{t:abstractconv1}
							If Assumption~\ref{a:trans}\ref{a:trans_avg} holds, then $x_n-T_{\gamma_n}x_n\to 0$. 
							\item
							\label{t:abstractconv3} 
							If  Assumption~\ref{a:trans}\ref{a:trans_avg}\&\ref{a:trans_bicont} hold, 
							then  both 
							$(x_n)_\nnn$ and $(T_{\gamma_n}x_n)_\nnn$ converge weakly to the same point in
							$\Fix T_{\overline{\gamma}}$.
						\end{enumerate}
					\end{theorem}
					\begin{proof}
						\ref{t:abstractconv0}: Let $C := \Fix T_{\overline{\gamma}}\neq\varnothing$. For arbitrary $c \in C$, set $c_0 := Q_{\gamma_0\gets \overline{\gamma}}c \in \Fix T_{\gamma_0}$, which holds true in view of Definition~\ref{d:Q-transport}\ref{a:trans_bijection}. Define a sequence $(c_n)_\nnn$ according to
						$(\forall\nnn) \; c_{n+1} := Q_{\gamma_{n+1}\gets \gamma_n}c_n. $ 
						From Definition~\ref{d:Q-transport}\ref{a:trans_bijection}, \ref{a:trans_cont} and~\ref{a:trans_semigroup}, we see that 
						\begin{equation*}\label{e:cn fix point}
							(\forall\nnn) \quad c_{n+1} = Q_{\gamma_{n+1}\gets \gamma_n}T_{\gamma_n}c_n
							\in\Fix T_{\gamma_{n+1}} 
							\;\text{and}\; c_n = Q_{\gamma_n\gets \gamma_0}c_{0}
							\to Q_{\overline{\gamma}\gets \gamma_0}c_{0} = c \in \Fix T_{\overline{\gamma}}. 
						\end{equation*}
						Thus, using Definition~\ref{d:Q-transport}\ref{a:trans_nonex-type} and nonexpansiveness of $T_{\gamma_n}$ yields
						\begin{equation}
							\begin{aligned}\label{e:decrease}
								\|x_{n+1}- c_{n+1}\| 
								&= 
								\|Q_{\gamma_{n+1}\gets \gamma_n}T_{\gamma_n}x_n - Q_{\gamma_{n+1}\gets \gamma_n}c_n\| \leq 
								\Lip_{\gamma_{n+1}\gets \gamma_n}\|T_{\gamma_n}x_n - c_n\|\\
								&= \Lip_{\gamma_{n+1}\gets \gamma_n}\|T_{\gamma_n}x_n - T_{\gamma_n}c_n\|\leq \Lip_{\gamma_{n+1}\gets \gamma_n}\|x_n-c_n\|.
							\end{aligned}
						\end{equation}
						Set $(\forall\nnn) \; \varepsilon_n = \Lip_{\gamma_{n+1}\gets \gamma_n} - 1 \geq 0$, so that $\sum_\nnn \varepsilon_n < \pinf$ due to \eqref{a:stepsize-series}. In view of Fact~\ref{f:RS}, $(\|x_n-c_n\|)_{\nnn}$ is convergent and hence there exists $L \geq 0$ such that $\|x_n-c_n\| \to L$. It follows that
						\[
						|\|x_n - c_n\| - \|c - c_n\|| \leq \|x_n - c\| \leq \|x_n - c_n\| + \|c_n - c\|
						\] which implies $\|x_n-c\| \to L$, as $c_n \to c$, showing that $(x_n)_\nnn$ is Opial with respect to $C$ as claimed.

						\ref{t:abstractconv1}:
						Suppose Assumption~\ref{a:trans}\ref{a:trans_avg} holds. Using $\alpha$-averagedness of the family $(T_\gamma)_{\gamma\in\Gamma}$ instead of nonexpansiveness in \eqref{e:decrease}, we obtain  
						\begin{align*}
							\|x_{n+1}-c_{n+1}\|^2 
							&\leq \Lip_{\gamma_{n+1}\gets \gamma_n}^2\|T_{\gamma_n}x_n - T_{\gamma_n}c_n\|^2 \\
							&\leq \Lip_{\gamma_{n+1}\gets \gamma_n}^2\left(\|x_n-c_n\|^2 - \frac{1-\alpha}{\alpha}\|T_{\gamma_n}x_{n}-x_n\|^2\right).
						\end{align*}
						Applying Fact~\ref{f:RS} yields $\sum_\nnn\Lip_{\gamma_{n+1}\gets \gamma_n}^2\|T_{\gamma_n}x_{n}-x_n\|^2<\pinf$. Since $(\forall\nnn)\;\Lip_{\gamma_{n+1}\gets \gamma_n} \geq 1$ from Definition~\ref{d:Q-transport}\ref{a:trans_nonex-type}, it follows that $\sum_\nnn\|T_{\gamma_n}x_{n}-x_n\|^2<\pinf$. Hence $x_n-T_{\gamma_n}x_n\to 0$ which yields \ref{t:abstractconv1}. 
						
						\ref{t:abstractconv3}: 
						Suppose that Assumptions~\ref{a:trans}\ref{a:trans_avg} and~\ref{a:trans_bicont} hold. In view of Fact~\ref{f:Opial}\ref{f:Opial-bounded}, $(x_n)_\nnn$ is bounded. Let $x$ be a weak cluster point of $(x_n)_\nnn$; say $x_{k_n}\weakly x$. 
						By \ref{t:abstractconv1}, $x_{k_n}-T_{\gamma_{k_n}}x_{k_n}\to 0$ and hence, by the parametric demiclosedness principle (Theorem~\ref{t:edemi}), we have $x \in \Fix T_{\overline{\gamma}}$. Since $x$ is an arbitrary cluster point, it follows that all weak cluster points of $(x_n)_\nnn$ lie in $ \Fix T_{\overline{\gamma}}$. Thus, from Fact~\ref{f:Opial}\ref{f:Opial-cluster}, it follows that $(x_n)_\nnn$ converges weakly to $x$, and \ref{t:abstractconv1} then implies $(T_{\gamma_n}x_n)_\nnn$ converges weakly to $x = T_{\overline{\gamma}}x$.
						
					\end{proof}

					\begin{remark}[On the dimension of $X$ and the parametric demiclosedness property]
						\label{r:weak-conv}
						Observe that for infinite-dimensional Hilbert spaces, Assumption~\ref{a:trans}\ref{a:trans_bicont} need not imply that the operator $(x,\gamma) \mapsto T_{\gamma}x$ is (sequentially) weakly continuous in $x$ (see, e.g., \cite[page 245]{zarantonello1971projections}). Thus the parametric demiclosedness principle (Theorem~\ref{t:edemi}) is needed to obtain the results in Theorem~\ref{t:abstractconv}\ref{t:abstractconv3}.
					\end{remark}
					
					\subsection{Variable stepsize DR algorithm} \label{s:variable-2-DR}
					
					In this section, we establish convergence of Algorithm~\ref{a:DR-2}, the relocated fixed-point iteration extension of the DR algorithm. First, we show that the operator defined in Proposition~\ref{p:Q} is a \Qname~for the DR operator. Given this example, we see that the construction in Definition~\ref{d:Q-transport} is natural given the resolvent properties we prove in Section~\ref{s:resolvents-and-nonexpansive}.

					\begin{lemma}[Fixed-point relocator for Douglas--Rachford] \label{l:example-Qname-DR}
						Given a nonempty set $\Gamma \subseteq \RPP$ and the family of DR operators $(T_{\gamma})_{\gamma \in \Gamma}$ on $X$ defined in \eqref{e:0222d}, $(\forall\gamma,\delta \in \Gamma)$ the operator  $Q_{\delta\gets \gamma}: X \to X$, such that
						\begin{equation*}
							Q_{\delta\gets \gamma} := \tfrac{\delta}{\gamma}\Id 
							+\big(1-\tfrac{\delta}{\gamma}\big)J_{\gamma A},
						\end{equation*} defines a \Qname~for $T_{\gamma}$ with Lipschitz constant $\Lip_{\delta\gets \gamma} = \max\{1, \frac{\delta}{\gamma}\}$.   
					\end{lemma}
					
					\begin{proof}
						Assertion \ref{a:trans_bijection} from  Definition~\ref{d:Q-transport} directly follows from Theorem~\ref{t:DRbijection}, and Assertion~\ref{a:trans_cont} follows from the fact that for fixed $\gamma \in \RPP$ and $x \in \Fix T_\gamma$, the map $\delta \mapsto Q_{\delta\gets\gamma}x$ is affine, thus continuous. For Assertion~\ref{a:trans_semigroup} , let $\gamma,\delta,\varepsilon\in\RPP$. Then, using \eqref{e:0222a} from Lemma~\ref{l:bijection}, we obtain 
						\begin{align*}
							Q_{\varepsilon\gets \delta}Q_{\delta\gets \gamma}
							&= \Big(\tfrac{\varepsilon}{\delta}\Id +
							\big(1-\tfrac{\varepsilon}{\delta}\big)J_{\delta A}\Big)
							\Big(\tfrac{\delta}{\gamma}\Id +
							\big(1-\tfrac{\delta}{\gamma}\big)J_{\gamma A}\Big)\\
							&= 
							\tfrac{\varepsilon}{\gamma}\Id +
							\big(\tfrac{\varepsilon}{\delta}
							-\tfrac{\varepsilon}{\gamma}\big)J_{\gamma A} 
							+\big(1-\tfrac{\varepsilon}{\delta}\big)J_{\delta A}%
							\Big(\tfrac{\delta}{\gamma}\Id +
							\big(1-\tfrac{\delta}{\gamma}\big)J_{\gamma A}\Big)\\
							&= 
							\tfrac{\varepsilon}{\gamma}\Id +
							\big(\tfrac{\varepsilon}{\delta}
							-\tfrac{\varepsilon}{\gamma}\big)J_{\gamma A} 
							+ \big(1-\tfrac{\varepsilon}{\delta}\big)J_{\gamma A}
							= 
							\tfrac{\varepsilon}{\gamma}\Id +
							\big(1-\tfrac{\varepsilon}{\gamma}\big)J_{\gamma A}\\
							&= Q_{\varepsilon\gets \gamma}.
						\end{align*}
						Finally, assertion~\ref{a:trans_nonex-type} holds due to Proposition~\ref{p:Q}.
						
					\end{proof}
					
					In Section~\ref{sec:resolvent-splitting}, we provide a number of additional examples of  \Qname{s} based on the one constructed in Lemma~\ref{l:example-Qname-DR}.
					
					In the next result, we show that the iteration scheme in \eqref{eq:RDR} and Algorithm~\ref{a:DR-2} are equivalent.
					
					\begin{proposition}[Equivalence between relocated fixed-point iteration and efficient implementation of DR] \label{p:DR-reformulation}
						Given $x_0 \in X$, the sequences $(x_n)_\nnn, (w_n)_\nnn, (y_n)_\nnn , (z_n)_\nnn \subseteq X $ are generated by Algorithm~\ref{a:DR-2} if and only if $(x_n)_\nnn$ and $(w_n)_\nnn$ conform to \eqref{eq:RDR} and $(\forall \nnn)$ $y_n = J_{\gamma_n B}R_{\gamma_n A}x_n$ and $z_n = J_{\gamma_{n-1}A}w_{n-1}$.
					\end{proposition}
					
					\begin{proof}
						
						It follows directly that \eqref{DR:variable-stepsize-2} is equivalent to the second identity in \eqref{eq:RDR}. To conclude, it suffices to show that $2z_n - x_n = R_{\gamma_n}x_n$.  From \eqref{e:0222a} and \eqref{DR:variable-stepsize-2}, it holds that $z_{n+1} = J_{\gamma_n A}w_{n} = J_{\gamma_{n+1} A}x_{n+1}$. The claim follows now from the definition of the reflectent.\end{proof}
					
					In the following, we prove a technical result which defines the conditions on the stepsizes that guarantee convergence of Algorithm~\ref{a:DR-2}.
					
					\begin{lemma} \label{l:stepsize-condition}
						Let $\Gamma \subseteq \RPP$ be nonempty and $(\gamma_n)_\nnn$ be a sequence such that
						\begin{equation}
							\label{a:stepsize-series-2}
							\gamma_{\text{low}}:= \inf_\nnn \gamma_n > 0 \text{~and~}\sum_\nnn (\gamma_{n+1}-\gamma_n)_+ < \pinf.
						\end{equation}  Then, the sequence $(\varepsilon_n)_\nnn$ defined by $(\forall\nnn) \; \varepsilon_n :=  \max\left\{0, \frac{\gamma_{n+1}}{\gamma_n} - 1\right\}$ satisfies $ \sum_\nnn \varepsilon_n < \pinf.$
						
					\end{lemma}
					
					\begin{proof}
						From the assumption, there exists $\gamma_{\text{low}} \in\RPP$ such that $(\forall\nnn) \; \gamma_{\text{low}} \leq \gamma_n $. Then $0\leq \varepsilon_n  \leq  \tfrac{(\gamma_{n+1} - \gamma_n)_+}{\gamma_{\text{low}}},$
						and hence \eqref{a:stepsize-series-2} implies $ \sum_\nnn \varepsilon_n < \pinf.$
					\end{proof}
					
					\begin{remark}[On the stepsize assumptions] \label{r:stepsize-cond} In the following, we comment on assumption \eqref{a:stepsize-series-2}. 
						
						\begin{enumerate}
							\item \label{r:stepsize-cond-i} If $(\gamma_n)_\nnn \subseteq \RR$ is a sequence such that $\inf_\nnn \gamma_n > -\infty$ and $\sum_\nnn (\gamma_{n+1}-\gamma_n)_+ < \pinf$, then $(\gamma_n)_\nnn$ is convergent. Indeed, let $\sigma_n :=\sum_{k =0}^n (\gamma_{k+1} -\gamma_k)_+$ and $\tau_n :=\gamma_{n+1} -\sigma_n$. Then $\lim_{n\to +\infty} \sigma_n =\sum_{n\in \mathbb{N}} (\gamma_{n+1}-\gamma_n)_+ <+\infty$ and $(\tau_n)_{n\in \mathbb{N}}$ is bounded from below. Moreover,
							\begin{align*}
								(\forall n\in \mathbb{N})\quad \tau_{n+1} -\tau_n =(\gamma_{n+2} -\gamma_{n+1}) -(\gamma_{n+2} -\gamma_{n+1})_+ \leq 0,    
							\end{align*}
							which means $(\tau_n)_{n\in \mathbb{N}}$ is decreasing and hence convergent. The latter implies the convergence of $(\gamma_n)_{n \in \mathbb{N}}$. In particular, any sequence satisfying \eqref{a:stepsize-series-2} is convergent.

							\item \label{r:stepsize-cond-3}
							
							Any increasing sequence of stepsizes $(\gamma_n)_\nnn \subseteq \RPP$ bounded from above can also be included in our analysis. Indeed, as $(\gamma_n)_\nnn \subseteq \RPP$ is convergent,  $\sum_\nnn (\gamma_{n+1}-\gamma_n)_+ = \sum_\nnn \gamma_{n+1}-\gamma_n = \lim_{n\to\pinf} \gamma_n - \gamma _1 < \pinf$. The conclusion then follows from Remark~\ref{r:stepsize-cond}\ref{r:stepsize-cond-i}.
							
							\item \label{r:stepsize-equiv-cond} 
							Assumption \eqref{a:stepsize-series-2} is equivalent to the stepsize condition in \cite[Theorem 3.2]{Lorenz-Tran-Dinh}, namely, \begin{equation} \label{a:alt-stepsize-cond}
								(\exists \gamma_{\text{low}}, \gamma_{\text{up}} \in \RPP ) \; \gamma_{\text{low}} \leq \gamma_n \leq \gamma_{\text{up}} \text{~~and~} \sum_\nnn|\gamma_{n+1}-\gamma_n| < \pinf.
							\end{equation} In fact, let \eqref{a:alt-stepsize-cond} hold. Clearly $\sum_\nnn (\gamma_{n+1}-\gamma_n)_+ \leq \sum_\nnn |\gamma_{n+1}-\gamma_n|$. Since \eqref{a:alt-stepsize-cond} implies  $(\gamma_n)_\nnn$ is a Cauchy sequence, the limit $\overline{\gamma}$ is bounded below by $\gamma_{\text{low}}\in\RPP$, and thus \eqref{a:stepsize-series-2} holds. Conversely, suppose \eqref{a:stepsize-series-2}  holds. Let $(a_n)_\nnn$ be given by \begin{equation*} 
								(\forall \nnn) \; a_n := \gamma_{n+1} - \gamma_n.\end{equation*} Hence, $(\forall \nnn) \;\gamma_{n+1} = \gamma_1 +\sum_{i=1}^n a_i$. Note that $(\forall \nnn)$ $(a_n)_+ = \frac{a_n + |a_n|}{2}$ and $\sum_{k=1}^n a_k = \gamma_{n+1} - \gamma_1$. Then $2 \sum_{k=1}^n (a_k)_+ = \gamma_{n+1} - \gamma_1 + \sum_{k=1}^n |a_k|$. In view of \eqref{a:stepsize-series-2}, $\gamma_{\text{low}} \leq \gamma_{n+1} = \gamma_1 + 2 \sum_{k=1}^n (a_k)_+ - \sum_{k=1}^n |a_k| $, which in turn implies $\sum_{k=1}^n |a_k| \leq \gamma_1 - \gamma_{\text{low}} + 2 \sum_{k=1}^n (a_k)_+$. Taking the limit as $n \to \pinf$, \eqref{a:stepsize-series-2} implies $\sum_{\nnn} |\gamma_{n+1} - \gamma_n| = \sum_{\nnn}|a_n| \leq \gamma_1 - \gamma_{\text{low}} + 2 \sum_{\nnn} (\gamma_{n+1} - \gamma_n)_+ < \pinf$. Therefore $(\gamma_n)_\nnn$ converges and $\gamma_{\text{up}} := \sup_\nnn \gamma_{n} < \pinf$, yielding \eqref{a:alt-stepsize-cond}.

						\end{enumerate}
						
					\end{remark}

					The following corollary is our main result regarding convergence of the relocated DR algorithm (Algorithm~\ref{a:DR-2}).
					\begin{corollary}[Convergence of relocated DR algorithm] \label{c:variable-2-DR}
						Let $\Gamma \subseteq \RPP$. Suppose $A, B$ are  maximally monotone operators on $X$ 
						such that $Z := \zer(A+B)\neq\varnothing$. Let  \begin{equation*}
							(\gamma_n)_\nnn \text{~in~} \Gamma \text{~be a sequence that 
								converges to~} \overline{\gamma}\in \Gamma  
							\text{~and~}  \sum_\nnn (\gamma_{n+1}-\gamma_n)_+ < \pinf . 
						\end{equation*} 
						Given $x_0  \in X$, generate the sequences $(x_n)_\nnn, (w_n)_\nnn, (y_n)_\nnn, (z_n)_\nnn $ with Algorithm~\ref{a:DR-2}. Then, $(x_n)_\nnn$ and $(w_n)_\nnn$ 
						converge weakly to some $x \in \Fix T_{\overline{\gamma}}$, where $T_{\gamma}$ is the DR operator in \eqref{e:0222d}, and 
						$(z_{n})_\nnn$ and $(y_{n})_\nnn$ converge weakly to $J_{\overline{\gamma} A}x\in Z$.
					\end{corollary}
					
					\begin{proof}
						First, $\Fix T_\gamma \neq \varnothing$ from \eqref{e:Fix}. Since $T_\gamma$ is firmly nonexpansive, Assumption~\ref{a:trans}\ref{a:trans_avg} holds with $\alpha=\tfrac{1}{2}$. Assumption~\ref{a:trans}\ref{a:trans_bicont} follows from Corollary~\ref{c:DR-op-bicont}. In view of Proposition~\ref{p:DR-reformulation}, \eqref{eq:RFP} holds for the \Qname{s} of $(T_\gamma)_{\gamma\in\Gamma}$ given in Lemma~\ref{l:example-Qname-DR}. Due to Lemma~\ref{l:stepsize-condition}, the Lipschitz constants of the \Qname{s} satisfy \eqref{a:stepsize-series}. Hence, it follows from Theorem~\ref{t:abstractconv} and $w_n = T_{\gamma_n}x_n$ that $x_n - w_n \to 0$, and $(x_n)_\nnn$ and $(w_n)_\nnn$ 
						converge weakly to some $x \in \Fix T_{\overline{\gamma}}$.
						
						Turning to the other sequences, we have $z_n = J_{\gamma_n A}x_n$, $y_n =  J_{\gamma_n B}(2 z_n - x_n)$ and $w_n =x_n -z_n +y_n$. Then $\frac{x_n - z_n}{\gamma_n} \in A z_n$ and $y_n \in B^{-1}(\frac{z_n - w_n}{\gamma_n})$. With this notation, $y_n -z_n = w_n - x_n \to 0$.  
						Moreover, rearranging terms, we have 
						\begin{equation*}
							\begin{pmatrix}
								\frac{x_n - w_n}{\gamma_n} \\ w_n - x_n
							\end{pmatrix} \in \left( \begin{bmatrix}
								A \\ B^{-1}
							\end{bmatrix} + \begin{bmatrix}
								0 & \Id \\ -\Id & 0
							\end{bmatrix}\right)\begin{pmatrix}
								z_n \\ \frac{z_n - w_n}{\gamma_n}
							\end{pmatrix}.
						\end{equation*}
						Skew symmetric linear operators are maximally monotone with full domain \cite[Example 20.35]{BC2017}, and the inverse of any maximally monotone operator is also maximally monotone \cite[Proposition 20.22]{BC2017}. Then, the operator above is maximally monotone (Fact~\ref{f:sum-monotone}), and hence demiclosed in view of Fact~\ref{f:maximal-demiclosed}. Since $(x_n)_\nnn$ is bounded, so is $(z_n)_\nnn$ due to Proposition~\ref{p:resnice}. Given a weak cluster point $z$ of $(z_n)_\nnn$ and $z_{k_n}\weakly z$, taking the limits $x_{k_n} \weakly x$, $w_{k_n} \weakly x$ and $\gamma_{k_n} \to \overline{\gamma}$ in the relation above yields
						\begin{equation*}
							\begin{pmatrix}
								0 \\ 0
							\end{pmatrix} \in \begin{bmatrix}
								A \\ B^{-1}
							\end{bmatrix}\begin{pmatrix}
								z \\ \frac{z - x}{\overline{\gamma}}
							\end{pmatrix} + \begin{pmatrix}
								\frac{z - x}{\overline{\gamma}} \\ -z
							\end{pmatrix}.
						\end{equation*} 
						Hence $x -z\in \overline{\gamma} A (z)$ and $z - x \in \overline{\gamma}B(z) $, so that $z \in Z$ and $z = J_{\overline{\gamma} A}(x)$. In other words, all cluster points of $(z_n)_\nnn$ coincide and $z_n \weakly J_{\overline{\gamma} A}x$. As $y_n -z_n\to 0$, we also have that $y_n \weakly J_{\overline{\gamma} A}x$.
					\end{proof}
					
					\begin{remark}[Relationship with non-stationary DR] \label{r:non-stat-DR} The non-stationary DR algorithm in  \cite[Eq. (12)]{Lorenz-Tran-Dinh}  is a particular case of Algorithm~\ref{a:DR-2}. Indeed, given $w_0 \in X$ and $\gamma_0 \in \RPP$, consider the sequences generated by \cite[Eq. (12)]{Lorenz-Tran-Dinh}: $(\forall\nnn)\;z_n = J_{\gamma_{n-1}A}w_{n-1}$, $\kappa_n = \tfrac{\|z_n\|}{\|z_n - w_{n-1}\|}$, $\gamma_n = \kappa_n \gamma_{n-1}$, $y_n = J_{\gamma_nB}((1+\kappa_n)z_n - \kappa_n w_{n-1})$, and $w_n = y_n + \kappa_n( w_{n-1} - z_n)$. By setting $x_n = \kappa_nw_{n-1} + (1 - \kappa_n)z_n$,  \eqref{DR:variable-stepsize-2} holds. Furthermore, since $x_{n} - z_{n} = \tfrac{\gamma_{n}}{\gamma_{n-1}}(w_{n-1} - z_{n})$, then $w_n = y_n + x_n - z_n$, and $y_n = J_{\gamma_n B}(z_n + \kappa_n(z_n - w_{n-1})) = J_{\gamma_n B}(2z_n - x_n)$, and thus \eqref{DR:variable-stepsize-1} holds, recovering Algorithm~\ref{a:DR-2}. Hence, in view of Remark~\ref{r:stepsize-cond}\ref{r:stepsize-equiv-cond}, Corollary~\ref{c:variable-2-DR} recovers the convergence results in \cite[Theorem 3.2]{Lorenz-Tran-Dinh}. Furthermore,   Corollary~\ref{c:variable-2-DR} resembles \cite[Corollary 3.3]{lorenz2025degenerate}, where the authors obtain convergence guarantees of a variable stepsize DR algorithm by interpreting the method as an instance of the degenerate variable metric proximal point
						algorithm \cite{bredies2022degenerate}.
					\end{remark}

					\begin{remark}[Disjoint fixed-point sets]
						\label{r:disjoint} 
						Let us stress that Theorem~\ref{t:abstractconv}---and thus also Corollary~\ref{c:NDR-convergence}---can lead to cases 
						where all fixed-point sets $\Fix T_\gamma$ are pairwise disjoint.
						Indeed, suppose that $X=\RR$ and consider the optimization problem 
						of minimizing $f(x)+g(x)$, where 
						$f = \iota_{\{1\}}$ and $g=-\ln$ (with $\dom g= \RPP$). 
						This problem has only 
						one feasible point, $z:=1$, which is also the minimizer of $f+g$. 
						Set $A=\partial f = N_{\{1\}}$ and $B=\partial g$. 
						Then $Z=\zer(A+B)=\{1\}$. 
						Now let $\gamma\in\RPP$. 
						First, $J_{\gamma A}= P_{\{1\}}$ is a constant operator  
						while 
						\begin{equation*}
							J_{\gamma B}x = \frac{x+\sqrt{x^2+4\gamma}}{2}
						\end{equation*}
						by, e.g., \cite[Example~24.40]{BC2017}. 
						Now $R_{\gamma A}x = 2z-x=2-x$ and so the Douglas--Rachford operator
						$T_\gamma$ satisfies 
						\begin{equation*} 
							T_{\gamma}x = x - 1 + \frac{(2-x)+\sqrt{(2-x)^2+4\gamma}}{2}
							= \frac{x+\sqrt{(2-x)^2+4\gamma}}{2}
							\;\;\text{and so}\;\; \Fix T_\gamma = \{1+\gamma\}. 
						\end{equation*}
						This motivates the need for the proposed notion of an Opial sequence  with respect to a sequence of sets that may be pairwise disjoint.
					\end{remark}

					\section{Resolvent splitting methods with variable stepsize} \label{sec:resolvent-splitting}

					In this section, we generalize our analysis in Section~\ref{s:variable-2-DR}, and develop a variable stepsize resolvent splitting framework to solve (maximally) monotone inclusions of the form
					\begin{equation} \label{e:N-inclusion}
						\text{find~~} x \in X \text{~~such that~~} 0 \in \sum_{i=1}^N A_ix,
					\end{equation} for $N \geq 2$, based upon the \emph{graph-based DR} framework introduced in \cite{bredies2024graph}. We show that the relocated fixed-point iteration associated with the graph-based DR method converges weakly to a fixed point of the underlying nonexpansive operator---in the sense of Theorem~\ref{t:abstractconv}---and then leverage this result to deduce variable stepsize extensions of some known resolvent splitting algorithms. 
					
					\subsection{Graph-based DR with variable stepsize}
					
					The resolvent splitting methods we analyze in this section are constructed using directed graphs, in which each node belongs to one of the operators in \eqref{e:N-inclusion}, and the arcs determine the order of evaluation of the resolvents  of each operator.  We follow \cite{bredies2024graph} to define the setting.
					
					Let $N \geq 2$ and $G = (\mN, \mE)$ be a connected directed graph such that \begin{equation} \label{e:diG-a1}
						\mN := \{1, \dots, N\}\text{~~and~~} (i,j) \in \mE \implies i <j.
					\end{equation} For $i \in \mN$, $d_i$ denotes the \emph{degree} of node $i$ which is the cardinality of the set of nodes \emph{adjacent} to $i$: $d_i := |\{j \in \mN: (i,j) \in \mE \text{~~or~~} (j,i) \in \mE\}|.$ Similarly, $d_i^+$ denotes the \emph{in-degree} of $i$: $d_i^+ := |\{j \in \mN:  (j,i) \in \mE\}|$, and  $d_i^-$ denotes the \emph{out-degree} of $i$: $d_i^- := |\{j \in \mN:  (i,j) \in \mE\}|.$ Clearly, $d_i = d_i^+ + d_i^-$. 
					Let $G^\prime = (\mN, \mE^\prime)$ be a connected subgraph of $G$ with the same set of nodes. The topology of $G^\prime$ defines an iteration operator, which we proceed to describe following \cite[Section 3]{bredies2024graph}. First, recall that the \emph{Laplacian} of $G^\prime$ is the matrix $L \in \RR^{N\times N}$  given by
					\begin{equation*}
						(\forall i,j \in \mN)\; L_{ij} := \begin{cases} 
							d_i & \text{~~if~~} i=j, \\
							-1 & \text{~~if~~} (i,j) \in \mE^\prime \text{~~or~~} (j,i) \in \mE^\prime, \\
							0 & \text{otherwise.}
						\end{cases}
					\end{equation*} Let $\Z \in \RR^{N \times (N-1)}$ define a so-called \emph{onto decomposition} of $L$, which is possible due to \cite[Proposition 2.1]{bredies2024graph}. That is, $\Z^*$ is a surjective matrix such that $L = \Z\Z^*$. Further, it also holds that $\ker(\Z^*) = \Span\{\bone\}$ \cite[Eq. (2.9)]{bredies2022degenerate}, where $\bone := (1,\dots,1)\in\RR^N$. In this manner, $\Z^*$ satisfies \cite[Assumption 1 (i)]{tam2024frugal}.  Given $k \geq 1$, denote by $I_k$ the identity map on $X^{k}$, and set $\mZ^* := \Z^* \otimes I_{1}$ where $\otimes$ denotes the Kronecker product.  We also adopt the notation $\bx \in X^{N-1}$ and  $\bz \in X^{N}$ with $\bx = (x_1,\dots, x_{N-1})$ and $\bz = (z_1,\dots, z_{N})$. In view of \cite[eq. (12)]{tam2024frugal}, \begin{equation} \label{e:Ker-Im}
						\ker(\mZ^*) = \{ \bx \in X^{N-1}: x_1 = \dots = x_{N-1}\} \text{~~and Im} (\mZ) = \{ \bz \in X^N: \sum_{i=1}^N z_i = 0\}.
					\end{equation}
					
					Define the operators $\C: X^{N-1} \to X^{2N-1}$ and $\mM: X^{2N-1} \to X^{2N-1}$ by
					\begin{equation} \label{e:graph-DR-M}
						\C^* := [\mZ^* \quad I_{N-1}], \text{~~and~~}\mM := \begin{bmatrix}
							\mL & \mZ\\ \mZ^* & I_{N-1}
						\end{bmatrix},
					\end{equation}
					where $\mathcal{L}:=L\otimes I_{1}$. By construction, we have $\mM =\C\C^*$. Moreover, let $(\forall i \in \mN)\; \overline{d}_i$ denote the degree of $i$ in the graph $(\mN, \mE \setminus \mE^\prime)$ and define the matrices $R, P \in \RR^{N\times N}$ as follows \begin{equation*}
						(\forall i,j \in \mN)\; R_{ij} := \begin{cases} 
							0 & \text{~~if~~} i=j \\
							L_{ij} & \text{~~if~~} i >j \\
							-L_{ij} & \text{otherwise}
						\end{cases} \text{~~and~~} P_{ij} := \begin{cases} 
							\overline{d}_{i} & \text{~~if~~} i=j \\
							-2 & \text{~~if~~} (j,i) \in \mE \setminus \mE^\prime \\
							0 & \text{otherwise}
						\end{cases}.
					\end{equation*}
					By denoting $\mR := R \otimes I_1$ and $\mP := P \otimes I_1$, set \begin{equation*}
						(\forall \gamma \in \RPP)\; \mA_{L,\gamma} := \gamma \text{diag}(A_1, \dots, A_N) + \mR + \mP,
					\end{equation*} where $\text{diag}(A_1, \dots, A_N): (x_1, \dots, x_N) \mapsto (A_1x_1, \dots, A_Nx_N)$ is the \emph{diagonal operator}, and \begin{equation*}
						\mA_{\gamma} := \begin{bmatrix}
							\mA_{L,\gamma} & - \mZ\\ \mZ^* & 0_{N-1}
						\end{bmatrix}
					\end{equation*} with $0_{N-1}$ being the zero operator in $X^{N-1}$. The operator $\mA_{\gamma}$ is maximally monotone \cite[Theorem 3.1]{bredies2024graph}.
					
					In this manner, given $\theta\in\left]0,2\right[$ and $\gamma \in \RPP $, the iteration operator $T_{\gamma}: X^{N-1} \to X^{N-1}$ for the graph-based DR method is given by \begin{equation} \label{e:graph-DR-T-1}
						(\forall \bx \in X^{N-1})\;T_\gamma\bx := (1-\theta)\bx+\theta J_{\C^*\vartriangleright \mA_\gamma}\bx \text{~~where~~} \C^*\vartriangleright \mA_\gamma := (\C^* \mA^{-1}_\gamma \C)^{-1}.
					\end{equation}
					Note that $J_{\C^*\vartriangleright \mA_\gamma}$ is well-defined as $\C^*\vartriangleright \mA_\gamma$ is maximally monotone due to \cite[Lemma 2.13]{bredies2022degenerate}. In view of \cite[Eqs. (3.9)--(3.10) \& Step 5 of Algorithm 3.1]{bredies2024graph}, it holds
					\begin{equation} \label{e:graph-DR-T}
						(\forall \bx \in X^{N-1})\; T_{\gamma}\bx = \bx - \theta\mZ^*\bz,
					\end{equation} 
					where $\bz = (z_1, \dots, z_N) \in X^N$ is given by \begin{equation} \label{e:graph-DR-z}
						(\forall i = 1, \dots, N)\; z_i := J_{\frac{\gamma}{d_i}A_i}\left(\tfrac{2}{d_i} \sum_{(h,i)\in\mathcal{E}}z_h + \tfrac{1}{d_i}\sum_{j=1}^{N-1} \Z_{ij}x_j\right).
					\end{equation}
					
					Having described the setting that builds the graph-based DR method of \cite{bredies2024graph}, we next collect some facts that sets this algorithm in the framework of Section~\ref{s:abstract-framework-convergence} with the goal of applying Theorem~\ref{t:abstractconv} to obtain convergence of the corresponding relocated fixed-point iteration.
					
					\begin{fact}\label{f:graph-DR}
						Let $G = (\mN, \mE)$ be a connected directed graph satisfying \eqref{e:diG-a1}, and $G^\prime = (\mN, \mE^\prime)$ be a connected subgraph of $G$. Let $A_1, \dots, A_N$ be $N$ maximally monotone operators on $X$, and $\gamma \in \RPP$. Then,
						\begin{enumerate}
							\item \label{f:graph-DR-degree}
							$d_1 = 1$, $\sum_{i=2}^N d_i^+ = \sum_{i=1}^{N-1} d_i^- = |\mE| $ and $\sum_{i=1}^N d_i = 2|\mE|$.
							
							\item \label{f:graph-DR-zeros}$(\forall \bz \in X^N)(\exists \bv \in X^{N-1})$  $(\bz, \bv) \in \zer(\mA_\gamma)$ if and only if $z:=z_1 = \dots = z_N$  solves \eqref{e:N-inclusion}.
							
							\item \label{f:graph-DR-T}
							$T_{\gamma}$ defined in \eqref{e:graph-DR-T-1} is $\theta/2$-averaged on $X^{N-1}$.
							\item \label{f:CdA-mm-2}
							$\Fix T_{\gamma} = \C^* \zer(\mA_\gamma) $.
						\end{enumerate}
					\end{fact}
					
					\begin{proof} 
						Item \ref{f:graph-DR-degree} follows directly from \eqref{e:diG-a1}. Item \ref{f:graph-DR-zeros} is shown in \cite[Theorem 3.1]{bredies2024graph}. For item \ref{f:graph-DR-T}, as $\C^*\vartriangleright \mA_\gamma $ is maximally monotone \cite[Lemma 2.3]{bredies2022degenerate}, its resolvent is $1/2$-averaged and $T_\gamma$ is $\theta/2$-averaged by Fact~\ref{f:averaged}\ref{f:averaged-scaling}.  As for item \ref{f:CdA-mm-2}, as argued in the proof of \cite[Theorem 2.14]{bredies2022degenerate}, $\Fix T_\gamma = \Fix J_{\C^*\vartriangleright \mA_\gamma} = \C^* \Fix (\mM + \mA_\gamma)^{-1}\mM$ and $\Fix (\mM + \mA_\gamma)^{-1}\mM = \zer(\mA_\gamma)$, from which the result follows.
					\end{proof}

					The next result is the basis of the fixed-point relocator for the graph-based DR algorithm.
					
					\begin{lemma} \label{l:system-FPR}
						Let $\gamma, \delta \in \RPP$, $\bx \in X^{N-1}$,  $\mZ = \Z \otimes I_{1}$,  and $z_1,\dots, z_N$ as defined in \eqref{e:graph-DR-z}. Set \begin{equation} \label{eq:def-e}
							\be(\bx) := \big( (d_i - 2d_i^+) z_i\big)_{i=1}^N - \frac{1}{N} \sum_{i=1}^N(d_i-2d_i^+)z_i\bone.
						\end{equation} Then, the following hold.
						\begin{enumerate}
							\item \label{l:system-FPR-1} $\be(\bx)\in \text{Im}(\mZ)$ and the unique solution to \begin{equation} \label{e:system-Q}
								\mZ \by =  \tfrac{\delta}{\gamma} \mZ \bx + \left(1 - \tfrac{\delta}{\gamma} \right) \be(\bx)
							\end{equation} is given by 
							\begin{equation}\label{eq:y pi}
								\by = \tfrac{\delta}{\gamma} \bx + \left(1 - \tfrac{\delta}{\gamma} \right) \mZ^{\dagger}\be(\bx),
							\end{equation} where $\mZ^{\dagger}$ denotes the pseudo-inverse of $\mZ$.
							\item \label{l:system-FPR-2}
							Let $T_\gamma$ be the operator defined in \eqref{e:graph-DR-T-1}.  If $\bx \in \Fix T_\gamma$, then $\by$ given by \eqref{eq:y pi} satisfies $ \by \in \Fix T_\delta$ and there exists $z \in X$ such that
							\begin{equation}\label{e:tilde-zeta}
								(\forall i \in \mN)\; z = z_i = J_{\frac{\delta}{d_i}A_i}\left(\tfrac{2}{d_i} \sum_{(h,i)\in\mathcal{E}}z_h + \tfrac{1}{d_i}\sum_{j=1}^{N-1} \Z_{ij}y_j\right).
							\end{equation}
						\end{enumerate}
					\end{lemma}
					
					\begin{proof}
						Let $\bx \in X^{N-1}$ and $\gamma, \delta \in \RPP$.
						\begin{enumerate}
							
							\item Using properties of the Kronecker product and the injectivity of $\Z$, it follows that the operator $\mZ $ is also injective. Moreover, \begin{align*}
								\sum_{i=1}^N \be(\bx)_i =  \sum_{i=1}^{N}(d_i-2d_i^+)z_i  - N\tfrac{1}{N}\sum_{i=1}^{N}(d_i-2d_i^+)z_i =0,
							\end{align*} and thus $\be(\bx) \in \text{Im}(\mZ)$ from \eqref{e:Ker-Im}. Then, from \cite[Proposition 3.30(ii)]{BC2017}, $\be(\bx) = P_{\text{Im}(\mZ)}\be(\bx) = \mZ\mZ^{\dagger}\be(\bx)$, and thus \eqref{eq:y pi} is a solution to \eqref{e:system-Q}. Uniqueness follows from injectivity of $\mZ$.
							
							\item Suppose $\bx \in \Fix T_\gamma$. From \eqref{e:graph-DR-T}, we have $\mZ^* \bz (\bx)= 0$ with $\bz(\bx) := (z_1,\dots, z_N)$. From \eqref{e:Ker-Im}, there exists $z \in X$ such that   \begin{equation} \label{e:Qname-graph-DR-proof}
								(\forall i \in \mN)\; z = z_i = J_{\frac{\gamma}{d_i}A_i}\left(\tfrac{2}{d_i} \sum_{(h,i)\in\mathcal{E}}z + \tfrac{1}{d_i}\sum_{j=1}^{N-1} \Z_{ij}x_j\right)=J_{\frac{\gamma}{d_i}A_i}\left(\tfrac{2 d_i^+}{d_i} z + \tfrac{1}{d_i}\sum_{j=1}^{N-1} \Z_{ij}x_j\right).
							\end{equation}
							Thus, using \eqref{e:Qname-graph-DR-proof} and Fact~\ref{f:graph-DR}\ref{f:graph-DR-degree}, we deduce \begin{equation}\label{e:e-tQ}
								\sum_{i=1}^N(d_i-2d_i^+)z_i = \sum_{i=1}^Nd_iz-2\sum_{i=2}^Nd_i^+z =2|\mathcal{E}|z-2|\mathcal{E}|z = 0.
							\end{equation} 
							Thus, from \eqref{eq:def-e} and $d_1^+ = 0$, we have $\be(\bx) = (d_1\zeta, (d_2 - 2d_2^+) \zeta, \dots ,( d_N - 2d_N^+) \zeta)$ and hence \eqref{e:system-Q} simplifies to 
							\begin{equation} \label{e:system-Q-2}
								\mZ\by =  \tfrac{\delta}{\gamma}  \mZ\bx + \left(1 - \tfrac{\delta}{\gamma} \right)(d_1\zeta, (d_2 - 2d_2^+) \zeta, \dots ,( d_N - 2d_N^+) \zeta).
							\end{equation} We need to prove that $\by \in \Fix T_\delta$, which, by \eqref{e:graph-DR-T},  \eqref{e:graph-DR-z} and \eqref{e:Ker-Im}, is equivalent to showing $\tilde{z}_1 = \dots= \tilde{z}_N$, where $$ (\forall i \in  \mN)\;\tilde{z}_i := J_{\frac{\delta}{d_i}A_i}\left(\tfrac{2}{d_i} \sum_{(h,i)\in\mathcal{E}}\tilde{z}_h + \tfrac{1}{d_i}\sum_{j=1}^{N-1} \Z_{ij}y_j\right).$$  We proceed by induction, showing that $(\forall i \in  \mN)\; \tilde{z}_i = z$, where $z$ is defined in \eqref{e:Qname-graph-DR-proof}. First, for the base case, note that combining \eqref{e:Qname-graph-DR-proof} for $i=1$ with \eqref{e:system-Q-2} and \eqref{e:0222a} yields
							\begin{align*}
								\tilde{z}_1  = J_{\frac{\delta}{d_1}A_1}\left(\tfrac{1}{d_1}\sum_{j=1}^{N-1} \Z_{1j}y_j\right) & = J_{\tfrac{\delta}{d_1}A_1}\left(\tfrac{\delta}{\gamma}\tfrac{1}{d_1}\sum_{j=1}^{N-1} \Z_{1j}x_j + \left(1 - \tfrac{\delta}{\gamma}\right)\zeta\right)\\ 
								&= J_{\tfrac{\delta}{d_1}A_1}\left(\tfrac{\delta}{\gamma}\tfrac{1}{d_1}\sum_{j=1}^{N-1} \Z_{1j}x_j+ \left(1 - \tfrac{\delta}{\gamma}\right)J_{\frac{\gamma}{d_1}A_1}\left( \tfrac{1}{d_1}\sum_{j=1}^{N-1} \Z_{1j}x_j\right)\right)\\
								&= J_{\frac{\gamma}{d_1}A_1}\left( \tfrac{1}{d_1}\sum_{j=1}^{N-1} \Z_{1j}x_j\right) =\zeta.
							\end{align*} Similarly, for $i = 2, \dots, N$, combining \eqref{e:system-Q-2} with \eqref{e:Qname-graph-DR-proof} and \eqref{e:0222a} yields
							\begin{align*}
								\tilde{z}_i &  = J_{\frac{\delta}{d_i}A_i}\left(\tfrac{2 d_i^+}{d_i}\zeta + \tfrac{1}{d_i}\sum_{j=1}^{N-1} \Z_{ij}y_j\right) \\
								& = J_{\tfrac{\delta}{d_i}A_i}\left(\tfrac{\delta}{\gamma} \left( \tfrac{2d_i^+}{d_i} \zeta + \tfrac{1}{d_i}\sum_{j=1}^{N-1} \Z_{ij}x_j\right) + \left(1 - \tfrac{\delta}{\gamma}\right)\zeta\right)\\
								& = J_{\tfrac{\delta}{d_i}A_i}\left(\tfrac{\delta}{\gamma} \left( \tfrac{2d_i^+}{d_i} \zeta + \tfrac{1}{d_i}\sum_{j=1}^{N-1} \Z_{ij}x_j\right)+ \left(1 - \tfrac{\delta}{\gamma}\right)J_{\frac{\gamma}{d_i}A_i}\left(\tfrac{2 d_i^+}{d_i} \zeta + \tfrac{1}{d_i}\sum_{j=1}^{N-1} \Z_{ij}x_j\right)\right)\\
								&=  J_{\frac{\gamma}{d_i}A_i}\left(\tfrac{2d_i^+}{d_i} \zeta + \tfrac{1}{d_i}\sum_{j=1}^{N-1} \Z_{ij}x_j\right) =  \zeta.
							\end{align*} 
							Hence, it follows that $(\forall i \in \mN)\; \tilde{z}_i = \zeta$, and thus $\by \in \Fix T_\delta$. Furthermore, observe that we  have also shown that $(\forall i \in \mN)\; z_i = \zeta = \tilde{z}_i$.\end{enumerate}\end{proof}

					In the following, we prove the existence of a fixed-point relocator for the graph-based DR algorithm.

					\begin{proposition}[Fixed-point relocator for graph-based DR] \label{p:Qname-graph-DR}
						Let $\gamma, \delta \in \RPP$. Consider the operator $Q_{\delta\gets \gamma}: X^{N-1} \to X^{N-1}$ given by\begin{equation*}
							(\forall \bx \in X^{N-1}) \; Q_{\delta\gets \gamma}\bx = \tfrac{\delta}{\gamma}\bx +  \left(1 - \tfrac{\delta}{\gamma} \right)\mZ^{\dagger}\be(\bx),
						\end{equation*} where $\be(\bx)$ is defined in \eqref{eq:def-e} and $\mZ^{\dagger}$ denotes the pseudo-inverse of $\mZ = \Z \otimes I_1$.  Then 
						$(Q_{\delta\gets \gamma})_{\gamma, \delta \in \RPP}$ defines  \Qname{s} for the graph-based DR operators $(T_{\gamma})_{\gamma\in\RPP}$ given by \eqref{e:graph-DR-T}-\eqref{e:graph-DR-z} with is Lipschitz constants \begin{equation} \label{e:graphDR-Lip}
							(\forall \gamma,\delta \in \RPP)\; \Lip_{\delta\gets \gamma}  = \max\left\{ 1, \tfrac{\delta}{\gamma} + \left|1-\tfrac{\delta}{\gamma}\right|\|\mZ^{\dagger}\|\sqrt{\sum_{i=1}^N \tfrac{(d_i - 2d_i^+)^2}{d_i^2}K_i^2}\right\}
						\end{equation}  where  $(\forall i \in \mN)\; K_i :=\left(2\sum_{(h,i)\in\mE}\tfrac{K_h}{d_h} +\sqrt{\sum_{j=1}^{N-1}\Z_{ij}^2}\right)$. 
					\end{proposition}
					
					\begin{proof}
						Let $\gamma,\delta,\varepsilon \in \RPP$, and $\bx \in \Fix T_\gamma$. In view of Lemma~\ref{l:system-FPR}\ref{l:system-FPR-2}, $\be(\bx) = \be(Q_{\delta\gets \gamma}\bx)$. Then \begin{equation}\label{e:group-graph-DR}\begin{aligned}     Q_{\varepsilon\gets \delta}Q_{\delta\gets \gamma}\bx & = \tfrac{\varepsilon}{\delta}Q_{\delta\gets \gamma}\bx + \left(1-\tfrac{\varepsilon}{\delta}\right)\mZ^{\dagger}\be(Q_{\delta\gets \gamma}\bx)\\
								& = \tfrac{\varepsilon}{\delta}\left( \tfrac{\delta}{\gamma}\bx + \left( 1 - \tfrac{\delta}{\gamma}\right)\mZ^{\dagger}\be(\bx)\right) + \left(1-\tfrac{\varepsilon}{\delta}\right)\mZ^{\dagger}\be(Q_{\delta\gets \gamma}\bx) \\
								& = \tfrac{\varepsilon}{\delta}\left( \tfrac{\delta}{\gamma}\bx + \left( 1 - \tfrac{\delta}{\gamma}\right)\mZ^{\dagger}\be(\bx)\right) + \left(1-\tfrac{\varepsilon}{\delta}\right)\mZ^{\dagger}\be(\bx) \\
								& =  \tfrac{\varepsilon}{\gamma}\bx + \left( 1 - \tfrac{\varepsilon}{\gamma}\right)\mZ^{\dagger}\be(\bx) \\
								& = Q_{\varepsilon\gets \gamma}\bx.
						\end{aligned}\end{equation}
						It remains to show that $Q_{\delta\gets \gamma}$ satisfies Definition~\ref{d:Q-transport}.
						\begin{enumerate}
							\item From Lemma~\ref{l:system-FPR}\ref{l:system-FPR-2}, $Q_{\delta\gets \gamma}\Fix T_\gamma \subseteq \Fix T_\delta$. Furthermore, 
							we claim that the inverse of $Q_{\delta\gets \gamma}: \Fix T_\gamma \to \Fix T_\delta$ is given by $Q_{\gamma\gets \delta}: \Fix T_\delta \to \Fix T_\gamma$. Indeed, \eqref{e:group-graph-DR} with $\varepsilon = \gamma$ and Remark~\ref{r:FPR-inverse-and-identity}  yield $(\forall\bx \in \Fix T_\gamma)\; Q_{\gamma\gets \delta}Q_{\delta\gets \gamma}\bx = Q_{\gamma\gets \gamma}\bx = \bx$. By interchanging the roles of $\gamma$ and $\delta$, $(\forall\by \in \Fix T_\delta)\; Q_{\delta\gets\gamma}Q_{ \gamma\gets\delta}\by = Q_{\delta\gets \delta}\by = \by$. This proves the claim, and thus Definition~\ref{d:Q-transport}\ref{a:trans_bijection} holds. 
							\item From construction, $(\forall \gamma \in \Gamma)(\forall \bx \in \Fix T_\gamma)$ $\delta \mapsto Q_{\delta\gets \gamma}\bx$ is affine, hence continuous. Hence Definition~\ref{d:Q-transport}\ref{a:trans_cont} holds.
							
							\item  Definition~\ref{d:Q-transport}\ref{a:trans_semigroup} follows from \eqref{e:group-graph-DR}.
							
							\item  Let $\bu, \bv \in X^{N-1}$ and $\delta,\gamma\in\RPP$. Then \begin{align*}Q_{\delta\gets \gamma}\bu - Q_{\delta\gets \gamma}\bv = \tfrac{\delta}{\gamma}(\bu-\bv) +  \left(1 - \tfrac{\delta}{\gamma} \right)\mZ^{\dagger}(\be(\bu)-\be(\bv)),\end{align*} 
							
							Since $\be(\bx) \in \text{Im}(\mZ)$, for $\overline{\be}(\bx) = (d_1 z_1 , (d_2 - 2d_2^+) z_2 , \dots ,(d_N - 2d_N^+) z_N)$, we have $\be(\bx) = P_{\text{Im}(\mZ)}\overline{\be}(\bx)$. Nonexpansiveness of the projection yields $$\|\be(\bu)-\be(\bv)\|^2 \leq  d_1^2\|z_1(\bu)-z_1(\bv)\|^2 + \sum_{i=2}^N(d_i - 2d_i^+)^2\|z_i(\bu)-z_i(\bv)\|^2$$
							From nonexpansiveness of $J_{\tfrac{\gamma}{d_1}A_1}$ and the Cauchy-Schwarz inequality, it follows that \begin{equation*} \|z_1(\bu) - z_1(\bv)\| \leq \tfrac{1}{d_1}\sum_{j=1}^{N-1}|\Z_{1j}|\|u_i-v_i\|\leq \tfrac{K_1}{d_1}\|\bu-\bv\| \end{equation*} Inductively, suppose for $i = 2, \dots, N$,  $(\forall j = 1, \dots, i-1)$, $\|z_j(\bu) - z_j(\bv)\| \leq \frac{K_j}{d_j}\|\bu-\bv\|$. Then \begin{align*}\|z_i(\bu) - z_i(\bv)\| &\leq \tfrac{2}{d_i}\sum_{(h,i)\in\mE}\|z_h(\bu) - z_h(\bv)\| + \tfrac{1}{d_i}\sum_{j=1}^{N-1}|\Z_{ij}|\|u_j-v_j\| \\& \leq \tfrac{K_i}{d_i} \|\bu-\bv\| .\end{align*} Then \begin{equation*}\|\be(\bu) - \be(\bv)\|\leq \sqrt{\sum_{i=1}^N \tfrac{(d_i - 2d_i^+)^2}{d_i^2}K_i^2}\|\bu-\bv\|,\end{equation*} from which \eqref{e:graphDR-Lip} follows.
						\end{enumerate}
					\end{proof}

					\begin{algorithm}[!ht]
						\caption{Variable stepsize graph-based resolvent splitting  for finding a zero of $\sum_{i=1}^NA_i$, $N\geq 2$. \label{a:DR-graph}}
						\SetKwInOut{Input}{Input}
						\Input{Compute $\mZ$ and its   pseudo-inverse $\mZ^{\dagger}$. Choose $ \bx_0 = (x_{0,1}, \dots, x_{0,N-1}) \in X^{N-1}$ and $\theta \in]0,2[$.}
						
						Define $z_{0,1} = J_{\gamma_1 A_1}(x_{0,1})$.
						
						\For{$n=0,1,\dots$}{
							Step 1. Intermediate step. Compute \begin{equation*} 
								\bw_n = \bx_n - \theta \mZ^*\bz_n
							\end{equation*} where $\bz_n = (z_{n,1}, \dots, z_{n,N}) \in X^{N}$ is given by \begin{equation*} 
								(\forall i = 1, \dots, N)\; z_{n,i} = J_{\frac{\gamma_n}{d_i}A_i}\left(\tfrac{2}{d_i} \sum_{(h,i)\in\mathcal{E}}z_{n,h} + \tfrac{1}{d_i}\sum_{j=1}^{N-1} \Z_{ij}x_{n,j}\right)
							\end{equation*}
							
							Step 2. Next iterate. Compute
							\begin{equation*} 
								\be_n = (d_1 z_{n,1} , (d_2 - 2d_2^+) z_{n,2} , \dots ,(d_N - 2d_N^+) z_{n,N}) - \frac{1}{N}\left(d_1z_ {n,1}+ \sum_{i=2}^N(d_i-2d_i^+)z_{n,i}\right)\bone
							\end{equation*}
							and 
							\begin{equation*} \bx_{n+1} = \tfrac{\gamma_{n+1}}{\gamma_n}\bw_n +  \left(1 - \tfrac{\gamma_{n+1}}{\gamma_n} \right)\mZ^{\dagger}\be_n.
							\end{equation*}
						}
					\end{algorithm}

					In Algorithm~\ref{a:DR-graph} we present the relocated  fixed-point version of the graph-based DR algorithm.  We are now ready to state the main convergence result of this section (cf. \cite[Theorem 3.2]{bredies2024graph}). 
					
					\begin{corollary}[convergence of relocated fixed-point graph-based DR] \label{c:graph-DR-convergence}
						Let $G = (\mN, \mE)$ be a connected directed graph satisfying \eqref{e:diG-a1}, and let $G^\prime = (\mN, \mE^\prime)$ be a connected subgraph of $G$. Let $A_1, \dots, A_N$ be  maximally monotone operators on $X$ such that $\zer(\sum_{i=1}^N A_i) \neq \varnothing$, and $\Gamma \subseteq \RPP$ be nonempty. Let  $\theta \in \left]0,2\right[$ and $(\gamma_n)_\nnn  $ in $ \Gamma$ be a sequence and $\overline{\gamma} \in \Gamma$ satisfying \eqref{a:stepsize-series}. 
						Given  $\bx_0 \in X^{N-1}$, consider the sequences $(\bx_n)_\nnn$, $(\bw_n)_\nnn$, and $(\bz_n)_\nnn$ generated by Algorithm~\ref{a:DR-graph}. 
						Then, $(\bx_n)_\nnn$ and $(\bw_n)_\nnn$  converge weakly to some $\bx \in \Fix T_{\overline{\gamma}}$  and $(\forall i \in \mN)$ the sequence $ (z_{n,i})_\nnn$ converges weakly to  $z \in \zer(\sum_{i=1}^N A_i)$ given by $z = J_{\tfrac{\overline{\gamma}}{d_1}A_1}\left(\tfrac{1}{d_1}\sum_{j=1}^{N-1} \Z_{1j}x_{j}\right)$.
						
					\end{corollary}
					
					\begin{proof}
						In view of Fact~\ref{f:graph-DR}\ref{f:graph-DR-T},  $(\forall \theta \in \left]0,2\right[)(\forall \gamma \in \Gamma)$ the operator $T_{\gamma}$ in \eqref{e:graph-DR-T} is $\frac{\theta}{2}-$averaged,  hence Assumption~\ref{a:trans}\ref{a:trans_avg} holds. In view of \eqref{e:graph-DR-z} and Proposition~\ref{p:resnice}, an inductive argument and the fact that sums and compositions of continuous functions are continuous show that  for $\Gamma \subseteq \RPP$, $(\forall i = 1, \dots, N)$ $(\bx,\gamma) \in X^{N-1} \times \Gamma \mapsto J_{\frac{\gamma}{d_i}A_i}\left(\tfrac{2}{d_i} \sum_{(h,i)\in\mathcal{E}}z_h + \tfrac{1}{d_i}\sum_{j=1}^{N-1} \Z_{ij}x_j\right)$ is continuous. It follows that $(\bx, \gamma) \mapsto T_\gamma \bx$ defined in \eqref{e:graph-DR-T} is continuous, hence Assumption~\ref{a:trans}\ref{a:trans_bicont} holds. Moreover, from Facts~\ref{f:graph-DR}\ref{f:graph-DR-zeros}\&\ref{f:CdA-mm-2} and $\zer(\sum_{i=1}^N A_i) \neq \varnothing$, $ \Fix T_\gamma =  \C^* \zer(\mA_\gamma) \neq \varnothing$. For the \Qname~in Proposition~\ref{p:Qname-graph-DR}, $(\forall\nnn)\;\bx_{n+1} = Q_{\gamma_{n+1}\gets \gamma_n}T_{\gamma_n}\bx_n$ and $\bw_n = T_{\gamma_n}\bx_n$. Then Theorem~\ref{t:abstractconv} implies that $\bx_n - \bw_n \to 0$, and  
						$(\bx_n)_\nnn$ and $(\bw_n)_\nnn$ converge weakly to some $\bx \in  \Fix T_{\overline{\gamma}}$.

						We now show  $(\forall i\in \mN) \; z_{n,i} \weak z$ for some point $z \in X$ to be defined. Let $s \in (\Span\{\bone\})^\perp$ be arbitrary. Then, since $\text{Im}(\Z) = \ker(\Z^*)^\perp = (\Span\{\bone\})^\perp$, there exists $t \in \RR^{N-1}$ such that $s = \Z t$ and
						\begin{equation} \label{e:sz-to-0}
							\sum_{i=1}^N s_iz_{n,i} = 
							\big( s^\top \otimes I_{N}\big)\bz_n = \big(  t^\top \Z^* \otimes I_{N}\big)\bz_n = \big(  t^\top  \otimes I_{N-1}\big)\mZ^*\bz_n  = \tfrac{1}{\theta}\big(  t^\top  \otimes I_{N-1}\big)(\bx_n - T_{\gamma_n}\bx_n)\to 0.
						\end{equation} 
						
						In particular, we have $ (\forall i,j \in \mN)\; z_{n,i} - z_{n,j} \to 0$. Using the definition of the resolvent together with \eqref{e:graph-DR-z} yields\begin{equation*}
							\begin{aligned}
								(\forall i \in \mN)\; \tfrac{d_i}{\gamma_n}\left( p_{n,i} - z_{n,i}\right)\in A_iz_{n,i} 
							\end{aligned}
						\end{equation*} where \begin{align*}
							(\forall i \in \mN)\; p_{n,i}  = \tfrac{2}{d_i} \sum_{(h,i)\in\mathcal{E}}z_{n,h} + \tfrac{1}{d_i}\sum_{j=1}^{N-1} \Z_{ij}x_{n,j}.
						\end{align*}
						Compactly, this system can be expressed as the inclusion
						\begin{equation}\label{e:demiclosed_}
							\hspace{-0.35cm}\mathcal{U}
							\begin{pmatrix}
								\tfrac{d_1}{\gamma_n}\left(p_{n,1} - z_{n,1}\right) \\
								\tfrac{d_2}{\gamma_n}\left(p_{n,2} - z_{n,2}\right) \\
								\vdots \\
								\tfrac{d_{N-1}}{\gamma_n}\left( p_{n,N-1} - z_{n,N-1}\right) \\
								z_{n,N}
							\end{pmatrix}  \ni
							\begin{pmatrix}
								z_{n,1}-z_{n,N} \\
								z_{n,2}-z_{n,N} \\
								\vdots \\
								z_{n,N-1}-z_{n,N}\\
								\sum_{i=1}^{N}\tfrac{d_i}{\gamma_n}\big(p_{n,i}-z_{n,i}\big)\\
							\end{pmatrix}
						\end{equation} for\begin{equation*}
							\mathcal{U} := \begin{pmatrix}
								A_1^{-1}\\ A_2^{-1} \\ \vdots \\ A_{N-1}^{-1} \\ A_N\\
							\end{pmatrix} + \begin{pmatrix}
								0 & 0 & \dots & 0 & -\Id \\
								0 & 0 & \dots & 0 & -\Id \\
								\vdots & \vdots & \ddots & \vdots & \vdots \\
								0 & 0 & \dots & 0 & -\Id \\
								\Id & \Id & \dots & \Id & 0 \\
							\end{pmatrix}.
						\end{equation*} As $(\bx_n)_\nnn$ is bounded, nonexpansiveness of the resolvent of $A_1$ implies $(z_{n,1})_\nnn$ is bounded, where $(\forall \nnn)\; z_{n,1} =J_{\frac{\gamma}{d_1}A_1}p_{n,1}$. Let $\check{z}$ be a weak cluster point of $(z_{n,1})_\nnn$, say $z_{k_n,1} \weak \check{z}$. Since $(\forall i = 2, \dots, N)\; z_{n,i-1} - z_{n,i} \to 0$, we have $ z_{k_n,i} \weakly \check{z}$. Consequently, for  \begin{align*}
							(\forall i \in \mN)\;p_i  = \tfrac{2d_i^+}{d_i} \check{z} + \tfrac{1}{d_i}\sum_{j=1}^{N-1} \Z_{ij}x_j,
						\end{align*} it holds that $p_{n,1} \weak p_1$ and, inductively, $(\forall i = 2, \dots, N)\; p_{k_n,i} \weak p_i$. Observe that, from the definition of $p_{n,i}$, \begin{equation}\label{e:aux-1}\begin{aligned}
								\sum_{i=1}^{N}\tfrac{d_i}{\gamma_n}\big(p_{n,i}-z_{n,i}\big) 
								& = \tfrac{1}{\gamma_n}\sum_{i=1}^{N}\sum_{j=1}^{N-1} \Z_{ij}x_{n,j} - \tfrac{d_1}{\gamma_n}z_{n,1} + \tfrac{1}{\gamma_n}\sum_{i=2}^{N} \left( 2 \sum_{(h,i)\in\mathcal{E}}z_{n,h} - d_iz_{n,i} \right) \\
								&= 0+0+ \tfrac{1}{\gamma_n}\sum_{i=1}^{N} \left( 2 \sum_{(h,i)\in\mathcal{E}}z_{n,h} - d_iz_{n,i} \right) \\
								&= \tfrac{1}{\gamma_n}\left(2\sum_{i=1}^N d_i^- z_{n,i} - \sum_{i=1}^N d_i z_{n,i}\right) 
								= \tfrac{1}{\gamma_n} \sum_{i=1}^N(d_i^- - d_i^+) z_{n,i} \to 0,
						\end{aligned}\end{equation}
						where the second equality follows from \eqref{e:Ker-Im} as $\sum_{i=1}^{N}\sum_{j=1}^{N-1} \Z_{ij}x_{n,j} = \sum_{i=1}^N (\mZ\bx_n)_i = 0$ combined with the fact that $d_1^+ = 0$, the third equality follows from the definition of $d_i^-$, and the last equality follows from the identity $d_i = d_i^+ + d_i^-$. The convergence to $0$ follows from \eqref{e:sz-to-0} and the fact that $ s \in \RR^N$, defined by $(\forall i \in \mN )\;s_i = d_i^- - d_i^+ $, belongs to $(\Span\{\bone\})^\perp$ in view of Fact~\ref{f:graph-DR}\ref{f:graph-DR-degree}. Taking the (weak) limit above along the subsequence $k_n \to \pinf$ yields $ \sum_{i=1}^{N}\tfrac{d_i}{\overline{\gamma}}\big(p_{i}-\check{z}\big) = 0.$ Further, the operator $\mathcal{U}$ in \eqref{e:demiclosed_} is the sum of a maximally monotone operator and a skew symmetric matrix, hence  maximally monotone (Fact~\ref{f:sum-monotone}) and  demiclosed (Fact~\ref{f:maximal-demiclosed}). By taking the limits $(\forall i = 1, \dots, N-1)$ $z_{k_n,i+1} - z_{k_n,i} \to 0$, $x_{k_n,i} \weakly x_i$ and $\gamma_{k_n} \to \overline{\gamma}$ in \eqref{e:demiclosed_}, and using \eqref{e:aux-1}, we obtain $(\forall i = 1, \dots, N-1)\;  A_i^{-1}\big(\tfrac{d_i}{\overline{\gamma}}(p_i - \check{z})\big) - \check{z} \ni 0$, and $A_N\check{z} + \sum_{i=1}^{N-1} \tfrac{d_i}{\overline{\gamma}}(p_i - \check{z}) \ni 0$. Then, $(\forall i = 1, \dots N - 1)$, $\check{z}  = J_{\tfrac{\overline{\gamma}}{d_i}A_i}p_i$ and, in view of \eqref{e:aux-1}, $\check{z}  = J_{\tfrac{\overline{\gamma}}{d_N}A_N}p_N$. Observe that $p_1$, from definition, is only determined by $\bx$, the unique limit of $(\bx_n)_\nnn$. Hence $\check{z} = J_{\tfrac{\overline{\gamma}}{d_1}A_1}p_1 $ is the unique cluster point of $(z_{n,1})_\nnn$, and  $z_{n,1} \weakly \check{z}$.  Since $(\forall i = 2, \dots, N)$ $z_{n,i-1} - z_{n,i}\to 0$, we have  $z_{n,i} \weakly \check{z}$. The result then follows for $z = \check{z}$.

					\end{proof}
					
					Having established the convergence of the relocated fixed-point iteration of graph-based DR splitting, we proceed to deduce a variable stepsize extension of the Malitsky--Tam resolvent splitting algorithm \cite{MT23}.

					\subsection{Variable stepsize Malitsky--Tam resolvent splitting for $N \geq 2$ operators}\label{s:MT}

					In \cite{MT23}, the authors propose a resolvent splitting algorithm to solve problem \eqref{e:N-inclusion} that only requires one resolvent evaluation per maximally monotone operator in each iteration. This algorithm is a special case of the graph-based DR algorithm \cite[Section 3.1]{bredies2024graph} by taking $\mE^\prime = \{(i,i+1): i = 1, \dots, N-1\} \text{~~and~~} \mE = \{(1,N)\} \cup \mE^\prime.$ In particular, $(\forall i =1, \dots, N)\; d_i = 2$,  $d_1^+ = 0$, $(\forall i =2, \dots, N-1)\; d_i^+ = 1$ and $d_N^+ = 2$. Hence, from \eqref{eq:def-e}, \begin{equation} \label{e-MT-form}
						(\forall \bx \in X^{N-1})\;\be(\bx) = (2 z_1, 0, \dots, 0, -2 z_N) - \frac{1}{N}\left(2 z_1 - 2 z_N\right)\bone
					\end{equation} As explained in \cite[Section 3.1]{bredies2024graph}, the onto decomposition of the Laplacian of a connected tree, such as $\mE^\prime$, is given by its incidence matrix $\Z$. For the Malitsky--Tam resolvent splitting algorithm,  the incidence matrix and its pseudo-inverse are given by:
					\begin{equation*} 
						\Z = \begin{bmatrix}
							1 & 0 & 0&\dots &0 & 0 \\
							-1 & 1 & 0 & \dots & 0 & 0 \\
							0 & -1 & 1 & \dots & 0 & 0 \\
							\vdots & \vdots  & \vdots & \ddots & \vdots & \vdots \\
							0 & 0 & 0 & \dots & -1 & 1 \\
							0 & 0 & 0 & \dots & 0 & -1 \\
						\end{bmatrix} \in \RR^{N\times(N-1)}
					\end{equation*} and \begin{equation} \label{e:MT-incidence}
						\Z^{\dagger} = \begin{bmatrix}
							1 - \frac{1}{N} & - \frac{1}{N} & - \frac{1}{N}& - \frac{1}{N}&\dots  & - \frac{1}{N} & - \frac{1}{N} \\
							1 - \frac{2}{N} & 1 - \frac{2}{N} & - \frac{2}{N}& - \frac{2}{N}& \dots  & - \frac{2}{N} & - \frac{2}{N} \\
							1 - \frac{3}{N} & 1 - \frac{3}{N} & 1 - \frac{3}{N}&  - \frac{3}{N}&\dots  & - \frac{3}{N}  & - \frac{3}{N} \\
							\vdots & \vdots & \vdots & \vdots & \dots & \vdots & \vdots \\
							1 - \frac{N-1}{N} & 1 - \frac{N-1}{N} & 1 - \frac{N-1}{N}&  - \frac{N-1}{N}&\dots  & 1 - \frac{N-1}{N}  & - \frac{N-1}{N}.
						\end{bmatrix} \in \RR^{(N-1) \times N}. 
					\end{equation}A direct application of \eqref{e:graph-DR-T}--\eqref{e:graph-DR-z} and the change of variables \begin{equation} \label{e:change-of-var}
						\gamma \leftarrow \tfrac{\gamma}{2}, \theta \leftarrow \tfrac{\theta}{2} \text{~and~} \bx \leftarrow \tfrac{1}{2}\bx 
					\end{equation} yields the \emph{Malitsky--Tam iteration operator} $T_\gamma$. For $\theta \in \zeroun$ and $\gamma \in\RPP$,  $T_{\gamma}: X^{N-1} \to X^{N-1}$ is given by 
					\begin{equation} \label{e:n-it-op-T}
						(\forall \bx \in X^{N-1})\; T_{\gamma}\bx = \bx + \theta\begin{pmatrix}
							z_2 - z_1\\ \vdots \\ z_N - z_{N-1}
						\end{pmatrix} 
					\end{equation} where \begin{equation} \label{e:n-it-op-z}
						\left\{\begin{aligned}
							z_1 &= J_{\gamma A_1}(x_1)& \\
							z_i &= J_{\gamma A_i}(z_{i-1} + x_i  - x_{i-1}) & \text{~for~} i = 2, \dots, N-1\\
							z_N &= J_{\gamma A_N}(z_1 + z_{N-1} - x_{N-1}).&
						\end{aligned}\right.
					\end{equation}

					In view of Proposition~\ref{p:Qname-graph-DR}, finding an appropriate \Qname~for the iteration operator defined in \eqref{e:graph-DR-T}--\eqref{e:graph-DR-z} amounts to solving the system of equations \eqref{e:system-Q}. In the following, we define a \Qname for this operator  by employing  Remark~\ref{r:FPR}.

					\begin{proposition}[Fixed-point relocator for Malitsky--Tam operator] \label{p:Qname-MT}
						For $\gamma, \delta \in \RPP$, consider the operator $\tilde{Q}_{\delta\gets \gamma} = (\tilde{Q}_{\delta\gets \gamma}^1,\dots, \tilde{Q}_{\delta\gets \gamma}^{N-1}): X^{N-1} \to X^{N-1}$ defined, for $(\forall \bx = (x_1,\dots,x_{N-1}) \in X^{N-1})$, as \begin{equation} \label{tQ:N-op}
							\left\{\begin{aligned}
								\tilde{Q}_{\delta\gets \gamma}^1\bx & = \left(\tfrac{\delta}{\gamma}\Id 
								+\big(1-\tfrac{\delta}{\gamma}\big)J_{\gamma A_1}\right)x_1  \\
								\tilde{Q}_{\delta\gets \gamma}^i\bx & = \tfrac{\delta}{\gamma}(x_i -x_1) 
								+\tilde{Q}_{\delta\gets \gamma}^{1}\bx, \text{~~for~~} i=2,\dots,N-1.\\
							\end{aligned}\right.
						\end{equation}  Then, $(\tilde{Q}_{\delta\gets \gamma})_{\gamma , \delta \in \RPP}$ defines  \Qname{s} for $(T_{\gamma})_{\gamma\in\RPP}$ given by \eqref{e:n-it-op-T}--\eqref{e:n-it-op-z}, with Lipschitz constants given by $\mathcal{\tilde{L}}_{\delta\gets\gamma}= \sqrt{\lambda_{\delta\gets\gamma}}$, where $\lambda_{\delta\gets\gamma}\geq 0$ is the maximum eigenvalue of the (symmetric) matrix $M_{\delta\gets\gamma} \in \RR^{(N-1)\times(N-1)}$ defined as \begin{equation*}
						M_{\delta\gets\gamma} = \begin{bmatrix}
							\max\{1,(\tfrac{\delta}{\gamma})^2\} + (N-2)\bigl(1-\tfrac{\delta}{\gamma}\bigr)^2 & \tfrac{\delta}{\gamma}|1 - \tfrac{\delta}{\gamma}| & \tfrac{\delta}{\gamma}|1 - \tfrac{\delta}{\gamma}|  & \dots & \tfrac{\delta}{\gamma}|1 - \tfrac{\delta}{\gamma}| \\
							\tfrac{\delta}{\gamma}|1-\tfrac{\delta}{\gamma}| & (\tfrac{\delta}{\gamma})^2 & 0 & \dots & 0 \\
							\tfrac{\delta}{\gamma}|1-\tfrac{\delta}{\gamma}| & 0 & (\tfrac{\delta}{\gamma})^2 & \dots & 0 \\
							\vdots & \vdots & \vdots & \cdots & \vdots
							\\
							\tfrac{\delta}{\gamma}|1-\tfrac{\delta}{\gamma}| & 0 & 0 & \dots & (\tfrac{\delta}{\gamma})^2
						\end{bmatrix}.
						\end{equation*}Furthermore, it holds that\begin{equation*}
						|\sqrt{\lambda_{\delta\gets\gamma}} - 1 | \leq \frac{(N-1)^2(\delta+\gamma)}{\gamma^2}|\delta - \gamma|.
						\end{equation*}
					\end{proposition}
					
					\begin{proof}
						
						Before the change of variables~\eqref{e:change-of-var}, Proposition~\ref{p:Qname-graph-DR} yields \Qname{s}  $(Q_{\delta\gets \gamma})_{\gamma,\delta \in\RPP}$ for $(T_{\gamma})_{\gamma\in\RPP}$ determined by \eqref{e-MT-form} and \eqref{e:MT-incidence}. In view of Remark~\ref{r:FPR}, we can examine \Qname{s} on fixed-point sets to yield new \Qname{s}. Let $\gamma \in \Gamma$ and $\bx \in \Fix T_\gamma$. From Lemma~\ref{l:system-FPR}\ref{l:system-FPR-2}, $z_1 = z_N = z = J_{\gamma A_1}x_1$, and thus \eqref{e-MT-form} yields $\be(\bx) = (2z, 0, \dots, 0, -2z) \in X^N$. Hence, from \eqref{e:MT-incidence}, $(\forall i =1, \dots, N-1)\; \big(\mZ^\dagger \be(\bx)\big)_i = 2z$, and thus $Q^i_{\delta\gets\gamma}\bx =\frac{\delta}{\gamma}x_i + \left(1 - \frac{\delta}{\gamma}\right)2z$. In particular, $Q^1_{\delta\gets\gamma}\bx =\frac{\delta}{\gamma}x_1 + \left(1 - \frac{\delta}{\gamma}\right)2z$, and, inductively, $(\forall i =2, \dots, N)$ $Q^i_{\delta\gets\gamma}\bx =\frac{\delta}{\gamma}(x_i - x_1) + Q^1_{\delta\gets\gamma}\bx$. Applying the change of variables in \eqref{e:change-of-var}, and in view of \eqref{tQ:N-op}, $Q_{\delta\gets\gamma} = \tilde{Q}_{\delta\gets\gamma}$ on $\Fix T_\gamma$. We next show $\tilde{Q}_{\delta\gets \gamma}$ is Lipschitz continuous. Take $\bx, \bw \in X^{N-1}$. It follows from Proposition~\ref{p:Q} that \begin{equation*}
							\|\tilde{Q}_{\delta\gets \gamma}^1\bx - \tilde{Q}_{\delta\gets \gamma}^1\bw\| \leq \max\left\{1,\tfrac{\delta}{\gamma}\right\}\|x_1 - w_1\|.
						\end{equation*} Furthermore, for $i = 2, \dots, N-1$, \eqref{tQ:N-op} and Fact~\ref{f:resolvents}\ref{f:resolvents-nonexpansive} yield \begin{align*}
							\|\tilde{Q}_{\delta\gets \gamma}^i\bx - \tilde{Q}_{\delta\gets \gamma}^i\bw \| & = \left\| \tfrac{\delta}{\gamma}(x_i - w_i) + \left( 1 - \tfrac{\delta}{\gamma}\right)(J_{\gamma A_1}x_1 - J_{\gamma A_1}w_1) \right\| \\
							& \leq \tfrac{\delta}{\gamma} \|x_i - w_i\| + \left| 1 - \tfrac{\delta}{\gamma}\right| \|x_1 - w_1\|.
						\end{align*} Hence, for $\zeta = (\|x_i - w_i\|)_{i=1}^{N-1}$, $\|\tilde{Q}_{\delta\gets\gamma}\bx - \tilde{Q}_{\delta\gets\gamma}\bw\|^2 \leq  \zeta^\top M_{\delta\gets\gamma} \zeta \leq \lambda_{\delta\gets\gamma}\|\bx-\bw\|^2$. Observe that $\lambda_{\delta\gets\gamma} = \max_{\|x\|=1} x^\top M_{\delta\gets\gamma}x \geq (M_{\delta\gets\gamma})_{11} \geq  1$. Therefore, $\tilde{Q}_{\delta\gets \gamma}$ is Lipschitz continuous with constant $\sqrt{\lambda_{\delta\gets\gamma}} \geq 1$, and thus $\tilde{Q}_{\delta\gets \gamma}$ is a \Qname{} from Remark~\ref{r:FPR}. Furthermore, from $|\sqrt{x}-1|\leq |x-1|$ $(\forall x \in \RPP)$, Weyl's inequality \cite[Theorem 8.1]{bhatia2007perturbation}, and the fact that $\vertiii{A} \leq \|A\|_F$ holds for any  matrix $A$ \cite[Theorem 5.6.9]{horn2012matrix}, where $\vertiii{\cdot}$ is the (matrix) operator norm and $\|\cdot\|_F$ is the Frobenius norm, we derive that $|\sqrt{\lambda_{\delta\gets \gamma}} - 1| \leq |\lambda_{\delta\gets \gamma} - 1| \leq \vertiii{M_{\delta\gets \gamma} - I} \leq \|M_{\delta\gets \gamma} - I\|_F$, where $I \in \RR^{(N-1)\times(N-1)}$ is the identity matrix. Hence, it suffices to bound the entries of $M_{\delta\gets \gamma} - I$. For the first diagonal entry, we have\begin{align*}
						(M_{\delta\gets \gamma})_{1,1}-1 &=\max\{1,(\tfrac{\delta}{\gamma})^2\} -1 + (N-2)\left(1-\tfrac{\delta}{\gamma}\right)^2 \\
						&= ((\tfrac{\delta}{\gamma})^2-1)_+ + \frac{N-2}{\gamma^2}(\delta-\gamma)^2 \leq \left|(\tfrac{\delta}{\gamma})^2-1\right| + \frac{N-2}{\gamma^2}(\delta-\gamma)^2 \\
						& =\frac{|\delta-\gamma|(\delta+\gamma)}{\gamma^2}  + \frac{N-2}{\gamma^2}(\delta-\gamma)^2  \leq \frac{(N-1)(\delta+\gamma)}{\gamma^2}|\delta - \gamma|.\end{align*} Next, let $i = 2, \dots, N-1$. For the rest of the diagonal entries, there holds \begin{equation*}
						|(M_{\delta\gets\gamma})_{i,i}-1| = \left|(\tfrac{\delta}{\gamma})^2 - 1\right| \leq \frac{\delta+\gamma}{\gamma^2}|\delta - \gamma|. \end{equation*} For the non-zero off diagonal entries, \begin{equation*}
						(M_{\delta\gets\gamma})_{i,1} = (M_{\delta\gets\gamma})_{1,i} =\tfrac{\delta}{\gamma}|1-\tfrac{\delta}{\gamma}| \leq  \tfrac{\delta}{\gamma^2}|\delta-\gamma|.\end{equation*} Hence,  for all $n \in \mathbb{N}$, \begin{equation*}
						\|M_{\delta\gets\gamma} - I\|_F \leq (N-1) \max_{i,j}|(M_{\delta\gets\gamma}-I)_{i,j}| \leq \frac{(N-1)^2(\delta+\gamma)}{\gamma^2}|\delta - \gamma|,
						\end{equation*} from which we conclude.
					\end{proof}
					
					In Algorithm~\ref{a:DR-N}, we introduce a variable stepsize resolvent splitting algorithm for finding a zero of $\sum_{i=1}^N A_i$, a variant of the Malitsky--Tam algorithm \cite{MT23}, stemming from the relocated fixed-point iteration associated with the \Qname~  introduced in Proposition~\ref{p:Qname-MT}. This \Qname~produces an algorithm that only requires one resolvent evaluation per maximally monotone operator in each iteration, just as in its original form, i.e. the method with constant stepsize. Although using the \Qname{s} $({Q}_{\delta\gets \gamma})_{\gamma, \delta\in\RPP}$ in Proposition~\ref{p:Qname-graph-DR} gives a convergent variable stepsize variant of the Malitsky--Tam algorithm, it results in a method that needs more resolvent evaluations than the original algorithm.

					\begin{algorithm}[!ht]
						\caption{Variable stepsize resolvent splitting  for finding a zero of $\sum_{i=1}^NA_i$, $N\geq 2$. \label{a:DR-N}}
						\SetKwInOut{Input}{Input}
						\Input{Choose $ x_{0,1}, \dots, x_{0,N-1} \in X$ and $\theta \in\zeroun$.}
						
						Define $z_{0,1} = J_{\gamma_1 A_1}(x_{0,1})$.
						
						\For{$n=0,1,\dots$}{
							Step 1. Intermediate step. Compute \begin{equation} \label{e:intermediate-N}
								\bw_n = \bx_n + \theta \begin{pmatrix}
									z_{n,2} - z_{n,1}\\ \vdots \\ z_{n,N} - z_{n,N-1}
								\end{pmatrix}
							\end{equation} where $z_{n,1}, \dots, z_{n,N} \in X$ are given by \begin{equation*} 
								\left\{\begin{aligned}
									z_{n,i} &= J_{\gamma_n A_i}(x_{n,i} + z_{n,i-1} - x_{n,i-1}) & \text{~for~} i = 2, \dots, N-1\\
									z_{n,N} &= J_{\gamma_n A_N}(z_{n,1} + z_{n,N-1} - x_{n,N-1}).&
								\end{aligned}\right.
							\end{equation*}
							
							Step 2. Next iterate. Compute
							\begin{equation} \label{DR:variable-stepsize-N} 
								\begin{cases}
									z_{n+1,1}  = J_{\gamma_n A_1}w_{n,1}\\
									x_{n+1,1}  = \dfrac{\gamma_{n+1}}{\gamma_n} w_{n,1} + \left( 1 - \dfrac{\gamma_{n+1}}{\gamma_n}\right) z_{n+1,1}
								\end{cases}
							\end{equation}For $i = 2, \dots, N-1$, 
							\begin{equation} \label{DR:variable-stepsize-N-b} 
								x_{n+1,i}  = \dfrac{\gamma_{n+1}}{\gamma_n} (w_{n,i} - w_{n,1}) + x_{n+1,1}.
							\end{equation}
						}
					\end{algorithm}
					
					In the following, we formalize the idea that Algorithm~\ref{a:DR-N} can be reformulated as a relocated fixed-point iteration of the operator defined in \eqref{e:n-it-op-T}--\eqref{e:n-it-op-z} for the \Qname~in Proposition~\ref{p:Qname-MT}. Then, we show convergence of Algorithm~\ref{a:DR-N}  using Corollary~\ref{c:graph-DR-convergence}. 
					
					\begin{proposition}[Equivalence between relocated fixed-point iteration and efficient implementation] \label{p:DR-N-reformulation}
						Given $\bx_0  \in X^{N-1}$ and $\theta \in \zeroun$, the sequences $(\bx_n)_\nnn, (\bw_n)_\nnn \subseteq X^{N-1}$ and $ (\bz_n)_\nnn \subseteq X^{N}$ are generated by Algorithm~\ref{a:DR-N} if and only if $(\bz_n)_\nnn$ conforms to \eqref{e:n-it-op-z} and \begin{equation*}
							\left\{\begin{aligned}
								\bw_n     &= T_{\gamma_n}\bx_n \\
								\bx_{n+1} &= \tilde{Q}_{\gamma_{n+1}\gets \gamma_n}\bw_n \\
							\end{aligned}\right. 
						\end{equation*} where the operators $(T_{\gamma})_{\gamma \in \RPP}$ are defined in \eqref{e:n-it-op-T}--\eqref{e:n-it-op-z} and $(\tilde{Q}_{\delta\gets \gamma})_{\gamma,\delta \in \RPP}$  in \eqref{tQ:N-op}.
					\end{proposition}

					\begin{proof}

						From \eqref{e:0222a} and \eqref{DR:variable-stepsize-N}, it follows that $ z_{n+1,1} = J_{\gamma_n A_1}(w_{n,1}) = J_{\gamma_{n+1} A_1}(x_{n+1,1}).$
						Therefore \eqref{e:n-it-op-z} holds for $(\bz_n)_\nnn$. In this manner, \eqref{e:n-it-op-z} 
						and \eqref{e:intermediate-N} yield $\bw_n     = T_{\gamma_n}\bx_n$. Moreover, \eqref{DR:variable-stepsize-N}  implies $x_{n+1,1} = \tilde{Q}^1_{\gamma_{n+1}\gets \gamma_n}(\bw_{n})$, and \eqref{DR:variable-stepsize-N-b} yields $(\forall i = 2, \dots, N)$ $x_{n+1,i}  = \tilde{Q}^i_{\gamma_{n+1}\gets \gamma_n}(\bw_{n})$.
						
					\end{proof}

					\begin{corollary}[Convergence of variable stepsize Malitsky--Tam resolvent splitting] \label{c:NDR-convergence}
						Given $N\geq 2$, suppose that $A_1, \dots, A_N$ are $N$ maximally monotone operators on $X$ 
						such that $\zer(\sum_{i=1}^N A_i)\neq\varnothing$.  Let $\Gamma\subseteq \RPP$ be nonempty, $\theta\in\zeroun$ and $(\gamma_n)_\nnn \text{~in~} \Gamma \text{~be a sequence that 
							converges to~} \overline{\gamma}\in \Gamma 
						\text{~and~}  \sum_\nnn (\gamma_{n+1}-\gamma_n)_+ < \pinf .$ Given $\bx_0 \in X^{N-1}$, consider the sequences $(\bx_n)_\nnn, (\bw_n)_\nnn \subseteq X^{N-1}$ and $ (\bz_n)_\nnn \subseteq X^{N}$ generated by Algorithm~\ref{a:DR-N}. Then, 
						both sequences $(\bx_n)_\nnn$ and $(\bw_n)_\nnn$ 
						converge weakly to some point $\bx \in \Fix T_{\overline{\gamma}}$, where $T_{\gamma}$ is the iteration operator defined in \eqref{e:n-it-op-T}--\eqref{e:n-it-op-z}, and
						$(\forall i \in \mN)$ $(z_{n,i})_\nnn$ converges weakly to $J_{\overline{\gamma} A_1}x_1\in Z$.
					\end{corollary}
					
					\begin{proof}
						First, $(\bx_n)_\nnn$ is a relocated fixed-point iteration (Proposition~\ref{p:DR-N-reformulation}) of the operator $T_{\gamma}$ with \Qname~given in \eqref{tQ:N-op}. Moreover, from Remark~\ref{r:stepsize-cond}\ref{r:stepsize-equiv-cond}, there exist  $\gamma_{\rm low}, \gamma_{\rm up} \in\RPP$ such that $(\forall\nnn) \; \gamma_{low} \leq \gamma_n \leq \gamma_{\rm up} $. Thus, from Proposition~\ref{p:Qname-MT}, $\sum_{\nnn}(\tilde{\Lip}_{\gamma_{n+1}\gets \gamma_n} -1 )\leq \frac{2(N-1)^2\gamma_{\rm up}}{\gamma_{\rm low}^2}\sum_{\nnn}|\gamma_{n+1} - \gamma_n| $. Then,  from Remark~\ref{r:stepsize-cond}\ref{r:stepsize-equiv-cond},  $\sum_\nnn(\tilde{\Lip}_{\gamma_{n+1}\gets \gamma_n} -1) < \pinf$. 
						From Corollary~\ref{c:graph-DR-convergence}, $(\bx_n)_\nnn$ converges weakly to some $\bx \in \Fix T_{\overline{\gamma}}$ and   $(\forall i = 1, \dots , N)$ $(z_{n,i})_\nnn$ converges weakly to $z\in \zer(\sum_{i=1}^N A_i)$ given in \eqref{e:tilde-zeta}. In particular, it holds that $z = z_1 =J_{\overline{\gamma} A_1}x_1$.
					\end{proof}

					\section{Final remarks} \label{s:conclusion}

					In this work, we proposed the relocated fixed-point iteration framework for a parametrized family of nonexpansive operators. An instance of this setting is the graph-based extension of the DR splitting algorithm, in which the parameter is the stepsize, yielding convergence of variable stepsize resolvent splitting methods.
					
					Among future research directions, the authors plan to investigate other splitting methods under the lens of the relocated fixed-point iteration framework. In particular, we are interested in exploring 
					three-operator splitting methods that allow forward and resolvent evaluations such as \cite{DavisYin17},  and the extensions examined in \cite{dao2025general,aragon2024forward,aakerman2025splitting}. We also plan to investigate implementations of the stepsize sequence $(\gamma_n)_{\nnn}$ beyond Remark~\ref{r:stepsize-cond}\ref{r:stepsize-cond-3}, and examine rates of convergence of resolvent splitting methods with variable stepsizes, similar to the approach developed in \cite{guler1991convergence} for the proximal point algorithm.

					\section*{Acknowledgments}
					The authors thank the mathematical research institute MATRIX in Australia where part of this research was performed.
					The research of FA, MND and MKT was supported in part by Australian Research Council grant DP230101749. 
					The research of HHB was supported in part by a Discovery Grant from the Natural Sciences and Engineering Research Council of Canada, and Australian Research Council grant DP230101749.
					
					\printbibliography
				\end{document}